\documentclass{article}
\usepackage{fullpage,url}
\usepackage{amsmath,amssymb,amsthm}
\newtheorem{theorem}{Theorem}
\newtheorem{proposition}{Proposition}
\newtheorem{lemma}{Lemma}
\theoremstyle{remark}
\newtheorem{remark}{Remark}
\newtheorem{example}{Example}

\usepackage{color}

\title{Polynomial Expressions of Carries in p-ary Arithmetics}
\author{Shizuo Kaji${}^1$, Toshiaki Maeno${}^2$, Koji Nuida${}^*{}^3{}^4$, Yasuhide Numata${}^5$}
\date{${}^1$ Yamaguchi University, Japan (\url{skaji@yamaguchi-u.ac.jp}) \\%
${}^2$ Meijo University, Japan (\url{tmaeno@meijo-u.ac.jp}) \\%
${}^3$ National Institute of Advanced Industrial Science and Technology (AIST), Japan (\url{k.nuida@aist.go.jp}) \\%
${}^4$ Japan Science and Technology Agency (JST) PRESTO Researcher \\%
${}^5$ Shinshu University, Japan (\url{nu@math.shinshu-u.ac.jp}) \\%
${}^*$ Corresponding author}

\newcommand{\Zplus}{\mathop{+_{\mathbb{Z}}}}
\newcommand{\Zminus}{\mathop{-_{\mathbb{Z}}}}
\newcommand{\Ztimes}{\mathop{\times_{\mathbb{Z}}}}
\newcommand{\Zsum}[2]{\sum_{#1}^{#2}{}_{\!\mathbb{Z}}\,}
\newcommand{\Zprod}[2]{\prod_{#1}^{#2}{}_{\!\mathbb{Z}}\,}

\begin{document}
\maketitle
\begin{abstract}
It is known that any $n$-variable function on a finite prime field of characteristic $p$ can be expressed as a polynomial over the same field with at most $p^n$ monomials.
However, it is not obvious to determine the polynomial for a given concrete function.
In this paper, we study the concrete polynomial expressions of the carries in addition and multiplication of $p$-ary integers.
For the case of addition, our result gives a new family of symmetric polynomials, which generalizes the known result for the binary case $p = 2$ where the carries are given by elementary symmetric polynomials.
On the other hand, for the case of multiplication of $n$ single-digit integers, we give a simple formula of the polynomial expression for the carry to the next digit using the Bernoulli numbers, and show that it has only $(n+1)(p-1)/2 + 1$ monomials, which is significantly fewer than the worst-case number $p^n$ of monomials for general functions.
We also discuss applications of our results to cryptographic computation on encrypted data.
%
\end{abstract}


\ \\
\noindent
{\large \textbf{Remark.}}
The authors are notified that the essential part of our Theorem \ref{thm:carry_for_p-ary_multiplication__in_introduction} appears (by a different approach) in: C.~Sturtivant, G.~S.~Frandsen, The Computational Efficacy of Finite-Field Arithmetic, Theoretical Computer Science \textbf{112} (1993) 291--309 (see Theorem 9.1(a) and Theorem 11.2 in that paper).
The authors deeply thank Akihiro Munemasa for the information.
The authors would like to keep this preprint online for reference purposes.

\section{Introduction}
\label{sec:intro}

A well-known but remarkable property of finite prime field $\mathbb{F}_p$ (where $p$ is a prime) is that, any function that computes a value in $\mathbb{F}_p$ from a tuple of elements of $\mathbb{F}_p$ can be expressed as a polynomial over $\mathbb{F}_p$.
Such a polynomial expression of a function can be taken to be of degree at most $p-1$ with respect to each variable (we call it a \lq\lq minimal polynomial expression''), hence the polynomial in $n$ variables consists of at most $p^n$ monomials and has total degree at most $n(p-1)$ in general.
Here we emphasize that, besides the general theory that guarantees the \emph{existence} of the minimal polynomial expression, it is of its own interest to determine such a \emph{concrete} expression of a given function, which may have a significantly smaller number of monomials than the general bound $p^n$ and/or a significantly lower total degree than the general bound $n(p-1)$.
In this paper, we study the explicit polynomial expressions of the carry functions in $p$-ary arithmetics (precisely, addition and multiplication of $p$-ary integers).
We also discuss applications to computation on encrypted data studied in cryptology, from which the present work is originally motivated.

\subsection{Our Problem and Results}
\label{subsec:intro_summary}

More precisely, we consider the following problem.
For $a \in \mathbb{F}_p$, we define $a_{\mathbb{Z}} \in \mathbb{Z}$ to be the representative of the residue class $a \in \mathbb{F}_p = \mathbb{Z}/p\mathbb{Z}$ chosen from the subset $[p-1] := \{0,1,\dots,p-1\}$ of $\mathbb{Z}$.
We sometimes write the addition, the subtraction and the multiplication operators in $\mathbb{Z}$ as $\Zplus$, $\Zminus$ and $\Ztimes$, respectively, for clarifying the distinction between the operators in $\mathbb{Z}$ and in $\mathbb{F}_p$.
We also use the symbols $\Zsum{}{}$ and $\Zprod{}{}$ in a similar manner.
Then we define functions $\varphi_i \colon (\mathbb{F}_p)^n \to \mathbb{F}_p$ for $i = 0,1,\dots$ by the following relation for $x_1,\dots,x_n \in \mathbb{F}_p$:
\begin{equation}
\label{eq:definition_for_carry_function_for_addition}
\Zsum{j=1}{n} (x_j)_{\mathbb{Z}} = \Zsum{i \geq 0}{} \bigl( \varphi_i(x_1,\dots,x_n)_{\mathbb{Z}} \Ztimes p^i \bigr) \enspace,
\end{equation}
i.e., the $p$-ary expression of the integer $(x_1)_{\mathbb{Z}} \Zplus \cdots \Zplus (x_n)_{\mathbb{Z}}$ is $(\dots,\varphi_1(x_1,\dots,x_n)_{\mathbb{Z}},\varphi_0(x_1,\dots,x_n)_{\mathbb{Z}})_p$.
For example, $\varphi_0(x,y)$ and $\varphi_1(x,y)$ represent the sum and the carry, respectively, for the $p$-ary addition of two single-digit values $x$ and $y$ (where the $p$-ary digits are naturally identified with elements of $\mathbb{F}_p$).
Similarly, we define functions $\psi_i \colon (\mathbb{F}_p)^n \to \mathbb{F}_p$ for $i = 0,1,\dots$ by the following relation for $x_1,\dots,x_n \in \mathbb{F}_p$:
\begin{equation}
\label{eq:definition_for_carry_function_for_multiplication}
\Zprod{j=1}{n} (x_j)_{\mathbb{Z}} = \Zsum{i \geq 0}{} \bigl( \psi_i(x_1,\dots,x_n)_{\mathbb{Z}} \Ztimes p^i \bigr) \enspace,
\end{equation}
i.e., the $p$-ary expression of the integer $(x_1)_{\mathbb{Z}} \Ztimes \cdots \Ztimes (x_n)_{\mathbb{Z}}$ is $(\dots,\psi_1(x_1,\dots,x_n)_{\mathbb{Z}},\psi_0(x_1,\dots,x_n)_{\mathbb{Z}})_p$.
In this setting, our problem is to determine the concrete minimal polynomial expressions of the functions $\varphi_i$ and $\psi_i$.
We note that, the definitions of $\varphi_i$ and $\psi_i$ imply immediately that
\begin{displaymath}
\varphi_0(x_1,\dots,x_n) = x_1 + \cdots + x_n \mbox{ and } \psi_0(x_1,\dots,x_n) = x_1 \cdots x_n
\end{displaymath}
(we emphasize that the right-hand sides are computed in $\mathbb{F}_p$ rather than $\mathbb{Z}$).
In the following argument, we focus on the other cases when $i \geq 1$.
We also note that, when $p = 2$, we have $\psi_i = 0$ for any $i \geq 1$ (since now $(x_j)_{\mathbb{Z}} \in \{0,1\}$).
In the following argument, we assume $p > 2$ for the case of multiplication operators.

For the carry functions $\varphi_i$ in the addition operators, when $p = 2$, a simple solution of the problem using elementary symmetric polynomials has been derived, e.g., by Boyar, Peralta and Pochuev \cite{BPP00} (see also Example \ref{ex:symmetric_polynomial_mod-2} in Section \ref{sec:polynomial_for_addition}).
We extend the result to the case of other primes $p$ and determine the minimal polynomial expressions of the functions $\varphi_i$, by using classical Lucas' Theorem \cite{Luc} in elementary number theory on congruent relations between some binomial coefficients.
Precisely, we prove the following result in Section \ref{sec:polynomial_for_addition}.
To state the result, we introduce a notation; for a positive integer $m$ and a (not necessarily reduced) fraction $a = \alpha/\beta \in \mathbb{Q}$ with $\alpha,\beta \in \mathbb{Z}$ and $\gcd(\beta,m) = 1$, we define $a^{\langle m \rangle} = \alpha \cdot \beta^{-1} \in \mathbb{Z}/m\mathbb{Z}$ where $\beta^{-1}$ means the inverse of $\beta$ in $\mathbb{Z}/m\mathbb{Z}$.
For example, $(5/66)^{\langle 7 \rangle} = 5 \cdot 5 = 4 \in \mathbb{F}_7$ since $66 \equiv 3 \pmod{7}$ and $3 \cdot 5 \equiv 1 \pmod{7}$.
We note that $a^{\langle m \rangle}$ is independent of a choice of such an expression $\alpha/\beta$ of $a$.
Then the result is as follows:

\begin{theorem}
\label{thm:p-ary_Hamming_weight__in_introduction}
For any index $i \geq 0$, the minimal polynomial expression of $\varphi_i$ is given by
\begin{displaymath}
\varphi_i(x_1,\dots,x_n)
= \sum_{d_1,\dots,d_n} \prod_{j=1}^{n} \left( \frac{ 1 }{ d_j! } \right)^{\langle p \rangle} x_j (x_j-1) \cdots (x_j-d_j+1)
\end{displaymath}
(see above for the notation $a^{\langle p \rangle}$), where the sum in the right-hand side is taken over all the $[p-1]$-restricted compositions $(d_1,\dots,d_n)$ of $p^i$ of length $n$, that is, tuples of $d_1,\dots,d_n \in [p-1]$ with $d_1 + \cdots + d_n = p^i$.
\end{theorem}
The polynomial in Theorem \ref{thm:p-ary_Hamming_weight__in_introduction} has total degree at most $p^i$, which is significantly lower than the above-mentioned bound $n(p-1)$ in many cases (note that, since $(x_1)_{\mathbb{Z}} \Zplus \cdots \Zplus (x_n)_{\mathbb{Z}} \leq n(p-1)$ for any $x_1,\dots,x_n \in \mathbb{F}_p$, the definition of $\varphi_i$ implies that $\varphi_i = 0$ unless $p^i \leq n(p-1)$).
The number of the terms is given by the extended binomial coefficients, namely, it is equal to the coefficient of $X^{p^i}$ in the polynomial $(1 + X + \cdots + X^{p-1})^n$.
As well as the known case $p = 2$, our polynomials for the case $p > 2$ are symmetric polynomials due to the symmetry of the addition.
On the other hand, in contrast to the case $p = 2$, these symmetric polynomials for $p > 2$ are somewhat complicated and no simple expressions in terms of famous generating families of symmetric polynomials (such as the elementary symmetric polynomials) are found so far.
Hence, this result yields a new family of symmetric polynomials; detailed studies of their properties are left as a future research topic.

Regarding the related work, we note that, the proof in the above-mentioned previous work \cite{BPP00} is specialized to the case $p = 2$ and is not straightforwardly applicable to a general $p$.
On the other hand, for the case $p > 2$, the minimal polynomial expression of the carry $\varphi_1(x,y)$ to the next digit for the addition of two $p$-ary values was recently derived by the third author and Kurosawa \cite{NK15}; but their proof is based on a case-by-case argument depending on the fact that the number of added values is two, and is not straightforwardly applicable to a general case either.
Our proof for the general case is different from the two previous results.

On the other hand, for the carry function $\psi_1$ to the next digit in the multiplication operators with $p > 2$ (note that the case $p = 2$ is trivial, as mentioned above), we determine a formula for the minimal polynomial expression of $\psi_1$ using the Bernoulli numbers.
(The other carry functions $\psi_i$ to higher digits, i.e., with $i \geq 2$, are not considered in this paper and are left as a future research subject.)
This result also yields another new family of symmetric polynomials.
More precisely, we prove the following result; here we use the convention $B_1 = -1/2$ (rather than $B_1 = 1/2$) for the Bernoulli numbers $B_{\ell}$, i.e., $t / (e^t - 1) = \sum_{m \geq 0} B_m t^m / m!$:

\begin{theorem}
\label{thm:carry_for_p-ary_multiplication__in_introduction}
Let $p$ be an odd prime.
Then the minimal polynomial expression of $\psi_1(x_1,\dots,x_n)$ is given by
\begin{displaymath}
\psi_1(x_1,\dots,x_n) = x_1 \cdots x_n \left( \Psi(x_1 \cdots x_n) - \sum_{j=1}^{n} \Psi(x_j) + (n-1) \Psi(1) \right) \enspace,
\end{displaymath}
where $\Psi(t)$ is a polynomial defined by
\begin{displaymath}
\Psi(t) = \sum_{i=1}^{p-2} \left( \frac{ B_{p-1-i} }{ p-1-i } \right)^{\langle p \rangle} t^i
= \sum_{i=1}^{(p-3)/2} \left( \frac{ B_{p-1-2i} }{ p-1-2i } \right)^{\langle p \rangle} t^{2i} + \frac{ p-1 }{ 2 } t^{p-2}
\end{displaymath}
(see above for the notation $a^{\langle p \rangle}$ for $a \in \mathbb{Q}$).
We also have
\begin{displaymath}
\Psi(1) = (w_p)^{\langle p \rangle} = \left( B_{p-1} + \frac{ 1 }{ p } - 1 \right)^{\langle p \rangle} \enspace,
\end{displaymath}
where $w_p = ( (p-1)! + 1 )/p$ is Wilson's quotient.
\end{theorem}

We note that, although $\varphi_1$ and $\psi_1$ in Theorems \ref{thm:p-ary_Hamming_weight__in_introduction} and \ref{thm:carry_for_p-ary_multiplication__in_introduction} look very different, these symmetric functions are related by $\psi_1(x+1,y) = \psi_1(x,y) + \varphi_1(xy,y)$ which is obvious from their meanings.
We emphasize that, the carry function $\psi_1(x_1,\dots,x_n)$ for the case of $n$ values is expressed as a sum of only $(n+1)(p-1)/2 + 1$ monomials, which is much fewer than the above-mentioned general bound $p^n$.
The number of monomials in $\psi_1$ is decreased further for some $p$; for example, the term $(n-1) \Psi(1)$ in $\psi_1$ vanishes if $w_p \equiv 0 \pmod{p}$, i.e., $p$ is an Wilson prime.
Examples of such primes are $p = 5$, $13$ and $563$, while it is still open whether or not $w_p \equiv 0 \pmod{p}$ for some other prime $p$.

\subsection{Motivation from Cryptology}

Here we explain the motivation of the present work from cryptology.
In the recent research area of cryptology, one of the most intensively studied topics is \emph{fully homomorphic encryption} (\emph{FHE}) \cite{Gen09}, which is an encryption scheme that enables \lq\lq computation on encrypted data''.
For example, in an FHE scheme recently proposed by the third author and Kurosawa \cite{NK15}, for any given ciphertexts $c_1,c_2$ which are encryption of (unknown) plaintexts $m_1,m_2 \in \mathbb{F}_p$, respectively, new ciphertexts corresponding to plaintexts $m_1 + m_2 \in \mathbb{F}_p$ and $m_1 \cdot m_2 \in \mathbb{F}_p$ can be generated from $c_1$, $c_2$ and some public parameters only, without knowing the secret plaintexts $m_1$ and $m_2$.
In other words, one can perform the addition and the multiplication operators for some data in an encrypted form while keeping the data secret.
By the fact on the polynomial expressions of functions mentioned above, this functionality is enough for generating a ciphertext corresponding to plaintext $f(m_1,m_2) \in \mathbb{F}_p$ for an \emph{arbitrary} function $f$.
This property gives rise to a problem of designing a concrete and efficient algorithm to compute the value of a given function over $\mathbb{F}_p$ by combining the addition and the multiplication only.
From the point of view, the results of this paper enable us to implement addition and multiplication of arbitrary-precision $p$-ary integers, where each digit of the integers is encrypted by the FHE scheme in \cite{NK15}.
Namely, for example, to calculate the carry in an addition of encrypted digits $x_1,\dots,x_n$, we compute the polynomial in Theorem \ref{thm:p-ary_Hamming_weight} below where the addition and the multiplication in $\mathbb{F}_p$ are replaced with the above-mentioned corresponding operations for the ciphertexts (note that subtraction operators in the polynomial over the finite field $\mathbb{F}_p$ can be replaced with suitable addition operators).
Such a concrete result, beyond just a theoretical possibility of such computation, is also new in the area of cryptology.


%

\subsubsection*{Acknowledgements.}
The authors thank Kaoru Kurosawa, and the members of Shin-Akarui-Angou-Benkyo-Kai, especially Shota Yamada, Keita Emura and Goichiro Hanaoka, for their precious comments on this work.
The authors also thank Go Yamashita for his insightful comments on this work, which yielded Remarks \ref{rem:p-ary_Hamming_weight_from_generating_function} and \ref{rem:relation_to_group_cohomology} and the Appendix.

\section{Preliminaries}
\label{sec:preliminary}

We summarize some notations and terminology used in this paper.
For any proposition $P(x)$ for an object $x$, let $\chi[P(x)]$ denote the characteristic function of $P(x)$, defined by
\begin{displaymath}
\chi[P(x)] = 1 \mbox{ if $P(x)$ is true, } \chi[P(x)] = 0 \mbox{ if $P(x)$ is false}.
\end{displaymath}
Let $p$ denote a prime number.
As mentioned in the Introduction, for $a \in \mathbb{F}_p$, we define $a_{\mathbb{Z}} \in \mathbb{Z}$ to be the representative of the residue class $a \in \mathbb{F}_p = \mathbb{Z}/p\mathbb{Z}$ chosen from the subset $[p-1] := \{0,1,\dots,p-1\}$ of $\mathbb{Z}$.
We sometimes write the addition, the subtraction and the multiplication operators in $\mathbb{Z}$ as $\Zplus$, $\Zminus$ and $\Ztimes$, respectively, for clarifying the distinction between the operators in $\mathbb{Z}$ and in $\mathbb{F}_p$.
We also use the symbols $\Zsum{}{}$ and $\Zprod{}{}$ in a similar manner.
For a polynomial $\varphi(x_1,\dots,x_n)$, let $\deg\varphi$ denote the total degree of $\varphi$, and let $\deg_{x_i}\varphi$ denote the degree of $\varphi$ with respect to the variable $x_i$.

For a function $f \colon (\mathbb{F}_p)^n \to \mathbb{F}_p$, we say that a polynomial $\varphi(x_1,\dots,x_n)$ over $\mathbb{F}_p$ is a \emph{polynomial expression} of $f$, if $\varphi(x_1,\dots,x_n) = f(x_1,\dots,x_n)$ for every tuple $(x_1,\dots,x_n) \in (\mathbb{F}_p)^n$.
The following fact is well-known; due to its importance in this paper, we give a proof of the fact for the sake of completeness.

\begin{proposition}
\label{prop:unique_polynomial_expression}
For any function $f \colon (\mathbb{F}_p)^n \to \mathbb{F}_p$, there exists a polynomial expression $\varphi$ of $f$ which has degree at most $p-1$ with respect to each variable.
Moreover, such a polynomial $\varphi$ is unique.
\end{proposition}
\begin{proof}
For the existence, for any $a = (a_1,\dots,a_n) \in (\mathbb{F}_p)^n$, Fermat's Little Theorem implies that the polynomial expression $\varphi_a$ of the function $\chi[x = a]$ ($x = (x_1,\dots,x_n)$) is given by $\varphi_a(x) = \prod_{i=1}^{n} (1 - (x_i - a_i)^{p-1})$.
Then the polynomial expression of a general $f$ is given by $\varphi(x) = \sum_{a \in (\mathbb{F}_p)^n} \varphi_a(x) f(a)$.

For the uniqueness, it suffices to consider the case of the zero function $f = 0$.
Assume, for the contrary, that there is such a non-zero polynomial $\varphi$.
When $n = 1$, this contradicts the polynomial remainder theorem.
When $n \geq 2$, by focusing on a non-zero coefficient (belonging to $\mathbb{F}_p[x_1,\dots,x_{n-1}]$) of some power of $x_n$ in $\varphi \in \mathbb{F}_p[x_1,\dots,x_{n-1}][x_n]$, the coefficient must be a polynomial expression of the zero function, therefore the argument is reduced to the case of smaller $n$.
Hence Proposition \ref{prop:unique_polynomial_expression} holds.
\end{proof}

We call the unique polynomial expression of the function $f$ as in Proposition \ref{prop:unique_polynomial_expression} the \emph{minimal polynomial expression} of $f$.
Then the following property also holds:

\begin{proposition}
\label{prop:polynomial_expression_is_minimal}
For any function $f \colon (\mathbb{F}_p)^n \to \mathbb{F}_p$, the minimal polynomial expression $\varphi$ of $f$ has the minimum total degree among all polynomial expressions of $f$.
\end{proposition}
\begin{proof}
For any polynomial expression $\psi$ of $f$, if $\deg_{x_i}\psi \geq p$ for some variable $x_i$, then $\psi$ can be converted to another polynomial expression of $f$ of lower degree with respect to $x_i$ by replacing $x_i{}^p$ with $x_i$, since $a{}^p = a$ for every $a \in \mathbb{F}_p$ by Fermat's Little Theorem.
Iterating the process, $\psi$ can be converted to the minimal polynomial expression of $f$, which is equal to $\varphi$ by the uniqueness property in Proposition \ref{prop:unique_polynomial_expression}.
Now the conversion process does not increase the total degree, therefore we have $\deg\varphi \leq \deg\psi$.
Hence Proposition \ref{prop:polynomial_expression_is_minimal} holds.
\end{proof}

We note that the minimal polynomial expression of any symmetric function is a symmetric polynomial owing to the uniqueness property, since any permutation of the variables in the polynomial also yields such a polynomial expression of the same function.
For any function over $\mathbb{F}_p$, we often identify the minimal polynomial expression of the function with the function itself unless some ambiguity occurs.

Now we introduce useful notations to regard some rational numbers as elements of $\mathbb{F}_p$.
For a positive integer $m$ and a (not necessarily reduced) fraction $a = \alpha/\beta \in \mathbb{Q}$ with $\alpha,\beta \in \mathbb{Z}$ and $\gcd(\beta,m) = 1$, we define
\begin{displaymath}
a^{\langle m \rangle} = \alpha \cdot \beta^{-1} \in \mathbb{Z}/m\mathbb{Z}
\end{displaymath}
where $\beta^{-1}$ means the inverse of $\beta$ in $\mathbb{Z}/m\mathbb{Z}$.
For example, $(5/66)^{\langle 7 \rangle} = 5 \cdot 5 = 4 \in \mathbb{F}_7$ since $66 \equiv 3 \pmod{7}$ and $3 \cdot 5 \equiv 1 \pmod{7}$.
We note that $a^{\langle m \rangle}$ is independent of a choice of such an expression $\alpha/\beta$ of $a$.
This implies that the map $a \mapsto a^{\langle m \rangle}$ is a ring homomorphism to $\mathbb{Z}/m\mathbb{Z}$ from the ring of rational numbers that can be expressed as a fraction $\alpha/\beta$ with $\alpha,\beta \in \mathbb{Z}$ and $\gcd(\beta,m) = 1$.
We restate this property for the sake of reference.
For any polynomial $F(x_1,\dots,x_n)$ over $\mathbb{Q}$ in which all coefficients can be expressed as fractions with denominators being coprime to $m$, we define $F^{\langle m \rangle}(x_1,\dots,x_n)$ to be the polynomial over $\mathbb{Z}/m\mathbb{Z}$ obtained by applying the map $a \mapsto a^{\langle m \rangle}$ to every coefficient.
Then we have the following, which we will use in our argument several times:

\begin{lemma}
\label{lem:rational_number_as_element_of_Fp}
Let $a_1,\dots,a_n \in \mathbb{Q}$, let $F(x_1,\dots,x_n)$ be a polynomial over $\mathbb{Q}$, and suppose that all of $a_1,\dots,a_n$ and all coefficients in $F$ can be expressed as fractions with denominators being coprime to $m$.
Then we have $F^{\langle m \rangle}(a_1^{\langle m \rangle},\dots,a_n^{\langle m \rangle}) = F(a_1,\dots,a_n)^{\langle m \rangle}$ (see above for the notations).
\end{lemma}

\section{Polynomial Expressions of Carries for Addition}
\label{sec:polynomial_for_addition}

In Section \ref{subsec:polynomial_for_addition__carry}, we determine the minimal polynomial expression of the function $\varphi_i(x_1,\dots,x_n)$ that yields the carry to the $i$-th digit in the integer addition $(x_1)_{\mathbb{Z}} \Zplus \cdots \Zplus (x_n)_{\mathbb{Z}}$ (see \eqref{eq:definition_for_carry_function_for_addition} in the Introduction for the precise definition of $\varphi_i$).
Then in Section \ref{subsec:polynomial_for_addition__algorithm_general}, we discuss algorithms for addition of $p$-ary integers where each step is composed of polynomial evaluations.

\subsection{The Results}
\label{subsec:polynomial_for_addition__carry}

Here we determine the minimal polynomial expressions of the functions $\varphi_i$ defined above.
Note that $\varphi_0(x_1,\dots,x_n) = \sum_{j=1}^{n} x_j$ (in $\mathbb{F}_p$), while we have $\varphi_i = 0$ if $n(p-1) < p^i$.
Our argument below is based on Lucas' Theorem \cite{Luc} in elementary number theory (see e.g., Exercise 6.a of Chapter 1 in \cite{Sta}):

\begin{proposition}
[{Lucas' Theorem \cite{Luc}}]
\label{prop:Lucas}
Let $a = (a_M \dots a_1a_0)_p$ and $b = (b_M \dots b_1b_0)_p$ be $p$-ary expressions of integers $a,b \geq 0$, where the leading digits are allowed to be zero.
Then we have
\begin{displaymath}
\binom{a}{b} \equiv \binom{a_M}{b_M} \cdots \binom{a_1}{b_1} \binom{a_0}{b_0} \pmod p \enspace,
\end{displaymath}
where we define $\binom{a'}{b'} = 0$ if $a' < b'$.
\end{proposition}

Then we have the following result (restatement of Theorem \ref{thm:p-ary_Hamming_weight__in_introduction} in the Introduction):

\begin{theorem}
\label{thm:p-ary_Hamming_weight}
For any index $i \geq 0$, the minimal polynomial expression of $\varphi_i$ is given by
\begin{displaymath}
\varphi_i(x_1,\dots,x_n)
= \sum_{d_1,\dots,d_n} \prod_{j=1}^{n} \left( \frac{ 1 }{ d_j! } \right)^{\langle p \rangle} x_j (x_j-1) \cdots (x_j-d_j+1)
\end{displaymath}
(see Section \ref{sec:preliminary} for the notation $a^{\langle p \rangle}$ for $a \in \mathbb{Q}$), where the sum in the right-hand side is taken over all the $[p-1]$-restricted compositions $(d_1,\dots,d_n)$ of $p^i$ of length $n$, that is, tuples of $d_1,\dots,d_n \in [p-1]$ with $d_1 + \cdots + d_n = p^i$.
\end{theorem}
\begin{proof}
First, we have
\begin{equation}
\label{eq:thm:p-ary_Hamming_weight__equality_from_Lucas_Theorem}
\varphi_i(x_1,\dots,x_n)_{\mathbb{Z}} = \binom{ \varphi_i(x_1,\dots,x_n)_{\mathbb{Z}} }{ 1 } \equiv \binom{ (x_1)_{\mathbb{Z}} \Zplus \cdots \Zplus (x_n)_{\mathbb{Z}} }{ p^i } \pmod{p}
\end{equation}
by Proposition \ref{prop:Lucas} applied to $a = (x_1)_{\mathbb{Z}} \Zplus \cdots \Zplus (x_n)_{\mathbb{Z}}$ and $b = p^i$ (i.e., $b_i = 1$ and $b_{i'} = 0$ for $i' \neq i$).
The binomial coefficient in the right-hand side is equal to the number of possible choices of $p^i$ objects from $(x_1)_{\mathbb{Z}} \Zplus \cdots \Zplus (x_n)_{\mathbb{Z}}$ objects.
We divide the $(x_1)_{\mathbb{Z}} \Zplus \cdots \Zplus (x_n)_{\mathbb{Z}}$ objects into $n$ blocks of $(x_1)_{\mathbb{Z}}$ objects, $(x_2)_{\mathbb{Z}}$ objects, ..., $(x_n)_{\mathbb{Z}}$ objects, and for each choice of the $p^i$ objects, we write the number of objects chosen from the $h$-th block as $d_h$.
Then the values $d_1,\dots,d_n$ satisfy that $d_h \in [p-1]$ (since $(x_h)_{\mathbb{Z}} \leq p-1$) and $d_1 + \cdots + d_n = p^i$, and we have
\begin{displaymath}
\binom{ (x_1)_{\mathbb{Z}} \Zplus \cdots \Zplus (x_n)_{\mathbb{Z}} }{ p^i }
= \sum_{d_1,\dots,d_n} \prod_{j=1}^{n} \binom{ (x_j)_{\mathbb{Z}} }{ d_j }
\end{displaymath}
where the sum is taken over all tuples $(d_1,\dots,d_n)$ as above.
Moreover, we have
\begin{displaymath}
\begin{split}
\binom{ (x_j)_{\mathbb{Z}} }{ d_j }^{\langle p \rangle}
&= \left( \frac{ (x_j)_{\mathbb{Z}} ((x_j)_{\mathbb{Z}} \Zminus 1) \cdots ((x_j)_{\mathbb{Z}} \Zminus d_j \Zplus 1) }{ d_j! } \right)^{\langle p \rangle} \\
&= \left( \frac{ 1 }{ d_j! } \right)^{\langle p \rangle} x_j (x_j - 1) \cdots (x_j - d_j + 1) \enspace.
\end{split}
\end{displaymath}
Since $(a_{\mathbb{Z}})^{\langle p \rangle} = a$ for any $a \in \mathbb{F}_p$, the claim of Theorem \ref{thm:p-ary_Hamming_weight} follows by summarizing these arguments.
\end{proof}

\begin{remark}
\label{rem:p-ary_Hamming_weight_from_generating_function}
The property \eqref{eq:thm:p-ary_Hamming_weight__equality_from_Lucas_Theorem} in the proof above can be also derived by comparing the coefficients of the monomial $X^{p^i}$ in the leftmost and the rightmost sides of the following equality for polynomials over $\mathbb{F}_p$:
\begin{displaymath}
\begin{split}
(1+X)^{(x_1)_{\mathbb{Z}} \Zplus \cdots \Zplus (x_n)_{\mathbb{Z}}}
& = (1+X)^{\varphi_0(x_1,\dots,x_n)_{\mathbb{Z}} \Zplus \varphi_1(x_1,\dots,x_n)_{\mathbb{Z}} \Ztimes p \Zplus \varphi_2(x_1,\dots,x_n)_{\mathbb{Z}} \Ztimes p^2 + \cdots} \\
&= (1+X)^{\varphi_0(x_1,\dots,x_n)_{\mathbb{Z}}} (1+X)^{\varphi_1(x_1,\dots,x_n)_{\mathbb{Z}} \Ztimes p} (1+X)^{\varphi_2(x_1,\dots,x_n)_{\mathbb{Z}} \Ztimes p^2} \cdots \\
&\equiv (1+X)^{\varphi_0(x_1,\dots,x_n)_{\mathbb{Z}}} (1+X^p)^{\varphi_1(x_1,\dots,x_n)_{\mathbb{Z}}} (1+X^{p^2})^{\varphi_2(x_1,\dots,x_n)_{\mathbb{Z}}} \cdots \pmod{p}
\end{split}
\end{displaymath}
(since $0 \leq \varphi_j(x_1,\dots,x_n)_{\mathbb{Z}} \leq p-1$ for each index $j$).
We note that Lucas' Theorem itself can be also proven by a similar argument.
\end{remark}

\begin{example}
\label{ex:symmetric_polynomial_mod-2}
When $p = 2$, the indices $d_1,\dots,d_n$ in the statement of Theorem \ref{thm:p-ary_Hamming_weight} are taken in such a way that $d_1,\dots,d_n \in \{0,1\}$ and $d_1 + \cdots + d_n = p^i$.
Then, by setting $S = \{j \in \{1,\dots,n\} \mid d_j = 1\}$, Theorem \ref{thm:p-ary_Hamming_weight} implies that
\begin{displaymath}
\varphi_i(x_1,\dots,x_n) = \sum_{S \subset \{1,\dots,n\} \,,\, |S| = 2^i} \prod_{j \in S} x_j = e_{2^i}(x_1,\dots,x_n) \enspace,
\end{displaymath}
i.e., $\varphi_i = e_{2^i}$, the elementary symmetric polynomial of degree $2^i$.
This coincides with the result by Boyar, Peralta and Pochuev \cite{BPP00} mentioned in the Introduction.
\end{example}

\begin{example}
\label{ex:symmetric_polynomial_mod-3}
When $p = 3$, the following expressions of the first three symmetric polynomials $\varphi_i$ in terms of some famous generating families of symmetric polynomials are calculated by using the software Sage, where $m_{\lambda}$, $e_j$ and $s_{\lambda}$ denote the monomial symmetric polynomials, elementary symmetric polynomials and Schur polynomials, respectively.
Here, some relations between these polynomials owing to the fact that the coefficient field is $\mathbb{F}_3$ instead of $\mathbb{Q}$ are utilized; e.g., we have $m_{1^1 3^1} = 2 m_{1^2} = - m_{1^2}$ as polynomials over $\mathbb{F}_3$.
\begin{displaymath}
\varphi_0 = m_{1^1} = e_1 \enspace.
\end{displaymath}
\begin{displaymath}
\begin{split}
\varphi_1
= m_{1^3} - m_{1^1 2^1} - m_{1^2}
= e_3 - e_2e_1 - e_2
= -s_{1^1 2^1} - s_{1^2} \enspace.
\end{split}
\end{displaymath}
\begin{displaymath}
\begin{split}
\varphi_2
& = m_{1^9} - m_{1^8} - m_{1^7 2^1} - m_{1^6} + m_{1^5 2^2} - m_{1^5 2^1} - m_{1^5} + m_{1^4 2^2} - m_{1^4 2^1} - m_{1^3 2^3} + m_{1^2 2^3} + m_{1^1 2^4} \\
& = e_9 + e_8e_1 - e_7e_2 + e_7 - e_6e_3 - e_6e_1 - e_6 + e_5e_4 + e_5e_3 - e_5e_1 - e_5 \\
&= (s_{1^9} - s_{1^5 2^2} + s_{1^1 2^4}) + (s_{1^8} + s_{1^6 2^1} + s_{1^4 2^2} + s_{1^2 2^3}) + (-s_{1^5 2^1}) + (s_{1^6} - s_{1^4 2^1}) + (-s_{1^5}) \enspace.
\end{split}
\end{displaymath}
\end{example}

We give an observation for the result of Theorem \ref{thm:p-ary_Hamming_weight}.
For a tuple $\vec{d} = (d_1,\dots,d_n)$ of non-negative integers, let
\begin{displaymath}
\Gamma_{\vec{d}}(x_1,\dots,x_n) = \prod_{j=1}^{n} x_j(x_j-1) \cdots (x_j - d_j + 1) \enspace.
\end{displaymath}
Then it is straightforward to show that, the linear space of polynomials in $x_1,\dots,x_n$ with total degree at most $D$ and degree at most $p-1$ in each variable $x_j$ is spanned as a basis (over any field) by the polynomials $\Gamma_{\vec{d}}(x_1,\dots,x_n)$ with $\vec{d} \in R_{\leq D}$, where $R_{\leq D}$ consists of tuples $\vec{d}$ with $d_j \in [p-1]$ for each index $j$ and $d_1 + \cdots + d_n \leq D$.
Let $R_D = R_{\leq D} \setminus R_{\leq D-1}$.
Now Theorem \ref{thm:p-ary_Hamming_weight} shows that the minimal polynomial expression of $\varphi_i$ lies in the subspace spanned by the polynomials $\Gamma_{\vec{d}}(x_1,\dots,x_n)$ with $\vec{d} \in R_{p^i}$, and the corresponding coefficients have a fairly simple expression.
This fact inspires an alternative proof of Theorem \ref{thm:p-ary_Hamming_weight} which does not rely on Lucas' Theorem (nor an essentially similar argument in Remark \ref{rem:p-ary_Hamming_weight_from_generating_function}); note that this proof is also different from the one in the previous work by Boyar et al.~\cite{BPP00} for $p = 2$.

\begin{proof}
[Another proof of Theorem \ref{thm:p-ary_Hamming_weight}]
First, we assume (as seen in the next paragraph) that $\deg\varphi_i(x_1,\dots,x_n) \leq p^i$.
Then $\varphi_i(x_1,\dots,x_n)$ belongs to the above-mentioned linear space over $\mathbb{F}_p$ spanned by $\Gamma_{\vec{d}}(x_1,\dots,x_n)$ with $\vec{d} \in R_{\leq p^i}$.
Let $\gamma_{\vec{d}}$ be the coefficient of $\Gamma_{\vec{d}}(x_1,\dots,x_n)$ in the corresponding expression of $\varphi_i(x_1,\dots,x_n)$.
Moreover, we define a partial ordering $\preceq$ on the tuples of $n$ non-negative integers in a way that $\vec{d} \preceq \vec{d'}$ if and only if $d_j \leq d'_j$ for every index $j$.
Now for $\vec{d},\vec{d'} \in R_{\leq p^i}$, we have $\Gamma_{\vec{d}}(d'_1,\dots,d'_n) = 0$ unless $\vec{d} \preceq \vec{d'}$, therefore
\begin{displaymath}
\varphi_i(d'_1,\dots,d'_n) = \sum_{\vec{d} \preceq \vec{d'}} \gamma_{\vec{d}} \cdot \Gamma_{\vec{d}}(d'_1,\dots,d'_n) \enspace.
\end{displaymath}
Based on this equality, since $\varphi_i(d'_1,\dots,d'_n) = 0$ for every $\vec{d'} \in R_{\leq p^i - 1}$ by the meaning of $\varphi_i$ and we have $\Gamma_{\vec{d}}(d_1,\dots,d_n) = \prod_{j=1}^{n} d_j! \neq 0$ in $\mathbb{F}_p$, a recursive argument implies that $\gamma_{\vec{d}} = 0$ for every $\vec{d} \in R_{\leq p^i - 1}$.
Moreover, by virtue of this property, for each $\vec{d} \in R_{p^i}$, we have
\begin{displaymath}
1 = \varphi_i(d_1,\dots,d_n) = \gamma_{\vec{d}} \cdot \Gamma_{\vec{d}}(d_1,\dots,d_n) = \gamma_{\vec{d}} \cdot \prod_{j=1}^{n} d_j! \enspace,
\end{displaymath}
therefore $\gamma_{\vec{d}} = \prod_{j=1}^{n} (1/d_j!)^{\langle p \rangle}$.
Hence $\varphi_i$ has the expression as in the statement of Theorem \ref{thm:p-ary_Hamming_weight}.

The remaining task is to show that $\deg\varphi_i(x_1,\dots,x_n) \leq p^i$.
The case $i = 0$ is obvious, therefore we consider the case $i \geq 1$.
We prove the claim by induction on $n$.
The first case $n = 1$ is obvious; $\varphi_i(x_1) = 0$ for $i \geq 1$.
On the other hand, for the case when $i = 1$ and $n = 2$, the fact $\deg\varphi_1(x_1,x_2) = p$ was proven in \cite{NK15} (by an elementary argument without Lucas' Theorem).
For the remaining cases, the $p$-ary expression of $(x_1)_{\mathbb{Z}} \Zplus \cdots \Zplus (x_{n-1})_{\mathbb{Z}}$ is $(\dots,\varphi_1(x_1,\dots,x_{n-1})_{\mathbb{Z}},\varphi_0(x_1,\dots,x_{n-1})_{\mathbb{Z}})_p$, and $\deg\varphi_i(x_1,\dots,x_{n-1}) \leq p^i$ by the induction hypothesis.
Now by the meaning of $\varphi_i$, we have $\varphi_i(x_1,\dots,x_n) - \varphi_i(x_1,\dots,x_{n-1}) \in \{0,1\}$, and the case $\varphi_i(x_1,\dots,x_n) - \varphi_i(x_1,\dots,x_{n-1}) = 1$ occurs precisely when $\varphi_j(x_1,\dots,x_{n-1}) = p-1$ for every $1 \leq j \leq i-1$ and $\varphi_0(x_1,\dots,x_{n-1})_{\mathbb{Z}} + (x_n)_{\mathbb{Z}} \geq p$.
For the former condition, $\deg\chi[y = p-1] \leq p-1$ by Proposition \ref{prop:unique_polynomial_expression}, therefore $\deg\chi[\varphi_j(x_1,\dots,x_{n-1}) = p-1] \leq p^j (p-1)$ for each $1 \leq j \leq i-1$.
On the other hand, for the latter condition, we have $\chi[\varphi_0(x_1,\dots,x_{n-1})_{\mathbb{Z}} + (x_n)_{\mathbb{Z}} \geq p] = \varphi_1(\varphi_0(x_1,\dots,x_{n-1}), x_n)$, therefore $\deg\chi[\varphi_0(x_1,\dots,x_{n-1})_{\mathbb{Z}} + (x_n)_{\mathbb{Z}} \geq p] \leq p$.
By these arguments, we have
\begin{displaymath}
\begin{split}
&\varphi_i(x_1,\dots,x_n) - \varphi_i(x_1,\dots,x_{n-1}) \\
&= \chi[\varphi_0(x_1,\dots,x_{n-1})_{\mathbb{Z}} + (x_n)_{\mathbb{Z}} \geq p] \cdot \prod_{j=1}^{i-1} \chi[\varphi_j(x_1,\dots,x_{n-1}) = p-1]
\end{split}
\end{displaymath}
and
\begin{displaymath}
\deg(\varphi_i(x_1,\dots,x_n) - \varphi_i(x_1,\dots,x_{n-1})) \leq p + \sum_{j=1}^{i-1} p^j (p-1) = p^i \enspace.
\end{displaymath}
Hence we have $\deg\varphi_i(x_1,\dots,x_n) \leq p^i$ by the induction hypothesis, concluding the proof.
\end{proof}

We also note that, when $n = 2$, Theorem \ref{thm:p-ary_Hamming_weight} can be refined as follows (note that now $\varphi_i = 0$ for $i \geq 2$, since $2(p-1) < p^2$):

\begin{theorem}
\label{thm:p-ary_Hamming_weight__two_numbers}
In the case $n = 2$, for $x_1,x_2 \in \mathbb{F}_p$, we have
\begin{displaymath}
\varphi_1(x_1,x_2)
= \sum_{d_1=1}^{p-1} (-1)^{d_1} \left( \frac{ 1 }{ d_1 } \right)^{\langle p \rangle} x_1 (x_1-1) \cdots (x_1-d_1+1) x_2 (x_2-1) \cdots (x_2-(p-d_1)+1) \enspace.
\end{displaymath}
\end{theorem}
\begin{proof}
First we note that $(p-1)! \equiv (-1)^p \pmod{p}$; indeed, when $p$ is odd, the set $\mathbb{F}_p \setminus \{-1,0,1\}$ can be divided into disjoint subsets of the form $\{\alpha,\alpha^{-1}\}$ with $\alpha \neq \alpha^{-1}$.
For the formula in Theorem \ref{thm:p-ary_Hamming_weight}, we have $d_2 = p - d_1$ for the indices $d_1,d_2$, therefore $1 \leq d_1 \leq p-1$.
Now we have
\begin{displaymath}
\left( \frac{ 1 }{ d_1!d_2! } \right)^{\langle p \rangle}
= \left( \frac{ (-1)^p (p-1)! }{ d_1! (p-d_1)! } \right)^{\langle p \rangle} \\
= \left( \frac{ (-1)^p }{ d_1 } \binom{p-1}{p-d_1} \right)^{\langle p \rangle}
= \left( \frac{ (-1)^{d_1} }{ d_1 } \right)^{\langle p \rangle}
\end{displaymath}
where we used the fact that $\binom{p-1}{a} \equiv (-1)^a \pmod{p}$ for any $a \in [p-1]$.
Therefore the claim holds by Theorem \ref{thm:p-ary_Hamming_weight}.
\end{proof}

\subsection{Addition of $p$-ary Integers Based on Polynomials}
\label{subsec:polynomial_for_addition__algorithm_general}

We show an algorithm for addition of $p$-ary integers $a_h = (a_{h,m} \dots a_{h,1} a_{h,0})_p$, $h = 1,\dots,n$, based on the result of Section \ref{subsec:polynomial_for_addition__carry}, which has applications to cryptology as mentioned in the Introduction.
Here, as above, each digit $a_{h,i}$ of $a_h$ is represented by an element of $\mathbb{F}_p$.
First, let $d$ be the smallest non-negative integer satisfying that $(n + d)(p-1) < p^{d+1}$.
Now we have
\begin{displaymath}
\begin{split}
a_1 + \cdots + a_n \leq n (p^{m+1}-1)
&= n (p-1)(p^m + \cdots + p + 1) \\
&< p^{d+1}(p^m + \cdots + p + 1)
< p^{d+1} \cdot p^{m+1}
= p^{m+d+2} \enspace,
\end{split}
\end{displaymath}
therefore the result of the addition $c = a_1 + \cdots + a_n$ can be expressed by $m+d+2$ digits; $c = (c_{m+d+1} \cdots c_1c_0)_p$, $c_i \in \mathbb{F}_p$.
Then the digits of $c$ and the carries $\gamma_{j,k} \in \mathbb{F}_p$ ($0 \leq j < k \leq m + d + 1$, $k \leq j + d$) during the addition ($\gamma_{j,k}$ means the carry to $k$-th digit from the calculation at $j$-th digit) are calculated by using the algorithm shown in Figure \ref{fig:algorithm_addition}.
Note that we have $\varphi_k(a_{1,i},\dots,a_{n,i},\gamma_{i-d,i},\gamma_{i-(d-1),i},\dots,\gamma_{i-1,i}) = 0$ for $k > d$ by the above-mentioned property $(n+d)(p-1) < p^{d+1}$.
This implies that the algorithm calculates the sum of $a_1,\dots,a_n$ correctly.
\begin{figure}
\centering
\caption{Algorithm for $p$-ary integer addition based on polynomials; here $d$ denotes the smallest non-negative integer satisfying $(n + d)(p-1) < p^{d+1}$}\medskip
\label{fig:algorithm_addition}
\fbox{%
\begin{minipage}{350pt}
\noindent
\texttt{Input:} $a_h = (a_{h,m} \dots a_{h,1} a_{h,0})_p$ ($h \in \{1,\dots,n\}$, $a_{h,i} \in \mathbb{F}_p$)\medskip

\noindent
\texttt{Initialize the variables} $\gamma_{j,k}$ \texttt{as} $\gamma_{j,k} \leftarrow 0$ \\
\texttt{For} $i = 0,1,\dots,m+d+1$ \texttt{Do:} \\
\hspace*{1em} \texttt{Set} $c_i \leftarrow \varphi_0(a_{1,i},\dots,a_{n,i},\gamma_{i-d,i},\gamma_{i-(d-1),i},\dots,\gamma_{i-1,i})$ \\
\hspace*{2.5em} \texttt{/*} \quad Comment: Input variables $a_{1,i},\dots,a_{n,i}$ are ignored when $i > m$ \quad \texttt{*/} \\
\hspace*{2.5em} \texttt{/*} \quad Comment: Input variables $\gamma_{i-j,i}$ are ignored when $i-j < 0$ \quad \texttt{*/} \\
\hspace*{1em} \texttt{For} $k = 1,2,\dots,\min\{d,m+d+1-i\}$ \texttt{Do:} \\
\hspace*{2em} \texttt{Set} $\gamma_{i,i+k} \leftarrow \varphi_k(a_{1,i},\dots,a_{n,i},\gamma_{i-d,i},\gamma_{i-(d-1),i},\dots,\gamma_{i-1,i})$ \\
\hspace*{1em} \texttt{End Do} \\
\texttt{End Do} \\
\texttt{Output} $c = (c_{m+d+1} \cdots c_1c_0)_p$
\end{minipage}
}
\end{figure}

From now, we focus on the case of addition of two integers (i.e., $n = 2$).
We note that, in this case, owing to the relation $2(p-1) + 1 < p^2$, it suffices to consider the carries from each digit to the next digit only, and the value of each carry is either $0$ or $1$.
Now the polynomials used in the algorithm above can be slightly simplified as follows:

\begin{proposition}
\label{prop:polynomial_for_carry__two_numbers}
For $x_1,x_2 \in \mathbb{F}_p$ and $\gamma \in \{0,1\} \subset \mathbb{F}_p$, we have $\varphi_1(x_1,x_2,\gamma) = \varphi'(x_1,x_2,\gamma)$, where
\begin{displaymath}
\varphi'(x_1,x_2,\gamma)
= \varphi_1(x_1,x_2) + \gamma \cdot (1 - (x_1 + x_2 + 1)^{p-1}) \enspace.
\end{displaymath}
\end{proposition}
\begin{proof}
In the calculation of $(x_1)_{\mathbb{Z}} \Zplus (x_2)_{\mathbb{Z}} \Zplus \gamma_{\mathbb{Z}}$, for each choice of $x_1,x_2$, the carry to the next digit for the case $\gamma = 1$ is different from that for the case $\gamma = 0$ if and only if $x_1 + x_2 = p-1$.
Moreover, in the case $x_1 + x_2 = p-1$, the carry is $1$ when $\gamma = 1$ and it is $0$ when $\gamma = 0$, i.e., it is equal to $\gamma$.
Since the carry when $\gamma = 0$ is nothing but $\varphi_1(x_1,x_2)$ for any $x_1,x_2$, we have
\begin{displaymath}
\varphi_1(x_1,x_2,\gamma) = \varphi_1(x_1,x_2) + \gamma \cdot \chi[x_1 + x_2 = p-1] \enspace,
\end{displaymath}
while we have $\chi[x_1 + x_2 = p-1] = 1 - (x_1+x_2+1)^{p-1}$ by Fermat's Little Theorem.
This completes the proof of Proposition \ref{prop:polynomial_for_carry__two_numbers}.
\end{proof}

Moreover, since $a_1 + a_2 \leq 2(p^{m+1}-1) \leq p(p^{m+1}-1) < p^{m+2}$, the sum $c = a_1 + a_2$ can be expressed by $m+2$ digits; $c = (c_{m+1} \cdots c_1c_0)_p$, $c_i \in \mathbb{F}_p$.
Now the addition of $a_1$ and $a_2$ can be calculated by the algorithm in Figure \ref{fig:algorithm_addition_two_values}.
\begin{figure}
\centering
\caption{Algorithm for addition of two $p$-ary integers based on polynomials}\medskip
\label{fig:algorithm_addition_two_values}
\fbox{%
\begin{minipage}{280pt}
\noindent
\texttt{Input:} $a_h = (a_{h,m} \dots a_{h,1} a_{h,0})_p$ ($h \in \{1,2\}$, $a_{h,i} \in \mathbb{F}_p$)\medskip

\noindent
\texttt{Set} $c_0 \leftarrow a_{1,0} + a_{2,0}$, $\gamma_{0,1} \leftarrow \varphi_1(a_{1,0},a_{2,0})$ \\
\texttt{For} $i = 1,\dots,m$ \texttt{Do:} \\
\hspace*{1em} \texttt{Set} $c_i \leftarrow a_{1,i} + a_{2,i} + \gamma_{i-1,i}$ \texttt{and} $\gamma_{i,i+1} \leftarrow \varphi'(a_{1,i},a_{2,i},\gamma_{i-1,i})$ \\
\texttt{End Do} \\
\texttt{Set} $c_{m+1} \leftarrow \gamma_{m,m+1}$ \\
\texttt{Output} $c = (c_{m+1} \cdots c_1c_0)_p$
\end{minipage}
}
\end{figure}


\section{Polynomial Expressions of Carries for Multiplication}
\label{sec:polynomial_for_multiplication}

In Section \ref{subsec:polynomial_for_multiplication__carry}, we determine the minimal polynomial expression of the function $\psi_1(x_1,\dots,x_n)$ that yields the carry to the next digit in the integer multiplication $(x_1)_{\mathbb{Z}} \Ztimes \cdots \Ztimes (x_n)_{\mathbb{Z}}$ (see \eqref{eq:definition_for_carry_function_for_multiplication} in the Introduction for the precise definition of $\psi_1$).
The other carry functions $\psi_i$ to higher digits, i.e., with $i \geq 2$, are not considered here and are left as a future research subject.
Here we assume $p > 2$, since the problem for the case $p = 2$ is trivial as mentioned in the Introduction (in fact, the assumption $p > 2$ is indeed used in our argument).
Then in Section \ref{subsec:polynomial_for_multiplication__algorithm}, we discuss an algorithm for multiplication of $p$-ary integers where each step is composed of polynomial evaluations.

\subsection{The Results}
\label{subsec:polynomial_for_multiplication__carry}

Here we determine the minimal polynomial expression of the function $\psi_1(x_1,\dots,x_n)$ defined above for $p > 2$.
The result is as follows (restatement of Theorem \ref{thm:carry_for_p-ary_multiplication__in_introduction} in the Introduction):

\begin{theorem}
\label{thm:carry_for_p-ary_multiplication}
Let $p$ be an odd prime.
Then the minimal polynomial expression of $\psi_1(x_1,\dots,x_n)$ is given by
\begin{equation}
\label{eq:carry_for_p-ary_multiplication__expression_by_Psi}
\psi_1(x_1,\dots,x_n) = x_1 \cdots x_n \left( \Psi(x_1 \cdots x_n) - \sum_{j=1}^{n} \Psi(x_j) + (n-1) \Psi(1) \right) \enspace,
\end{equation}
where $\Psi(t)$ is a polynomial defined by
\begin{equation}
\label{eq:carry_for_p-ary_multiplication__Psi}
\Psi(t) = \sum_{i=1}^{p-2} \left( \frac{ B_{p-1-i} }{ p-1-i } \right)^{\langle p \rangle} t^i
= \sum_{i=1}^{(p-3)/2} \left( \frac{ B_{p-1-2i} }{ p-1-2i } \right)^{\langle p \rangle} t^{2i} + \frac{ p-1 }{ 2 } t^{p-2}
\end{equation}
(see Section \ref{sec:preliminary} for the notation $a^{\langle p \rangle}$ for $a \in \mathbb{Q}$).
We also have
\begin{equation}
\label{eq:carry_for_p-ary_multiplication__Psi_constant}
\Psi(1) = (w_p)^{\langle p \rangle} = \left( B_{p-1} + \frac{ 1 }{ p } - 1 \right)^{\langle p \rangle} \enspace,
\end{equation}
where $w_p = ( (p-1)! + 1 )/p$ is Wilson's quotient.
\end{theorem}

We recall that we are using the convention $B_1 = -1/2$ (rather than $B_1 = 1/2$) for the Bernoulli numbers $B_{\ell}$, i.e., $t / (e^t - 1) = \sum_{m \geq 0} B_m t^m / m!$.
By this and the fact that $B_{\ell} = 0$ for odd indices $\ell > 1$, the second equality in \eqref{eq:carry_for_p-ary_multiplication__Psi} follows immediately from the first equality.
On the other hand, the second equality in \eqref{eq:carry_for_p-ary_multiplication__Psi_constant} is nothing but the following known relation \cite{Gla}: $w_p \equiv B_{p-1} + 1/p - 1 \pmod{p}$ for any prime $p$.

We divide the remaining proof of Theorem \ref{thm:carry_for_p-ary_multiplication} into the following three steps:

\begin{lemma}
\label{lem:carry_for_p-ary_multiplication__step1}
In the situation of Theorem \ref{thm:carry_for_p-ary_multiplication}, if $n = 2$, then the function $\psi_1(x_1,\dots,x_n)$ can be written as \eqref{eq:carry_for_p-ary_multiplication__expression_by_Psi} for some polynomial $\Psi(t)$ of degree at most $p-2$ with no constant term.
\end{lemma}
\begin{proof}
By Proposition \ref{prop:unique_polynomial_expression}, we can write $\psi_1(x,y)$ uniquely as $\psi_1(x,y) = \sum_{i,j=0}^{p-1} \alpha_{i,j} x^i y^j$ with $\alpha_{i,j} \in \mathbb{F}_p$.
Note that $\alpha_{i,j} = \alpha_{j,i}$, since the multiplication is symmetric.
From now, we investigate the coefficients $\alpha_{i,j}$.

First, note that $\psi_1(x,y) = 0$ if $y = 0$.
This implies that $\psi_1(x,0) = \sum_{i=0}^{p-1} \alpha_{i,0} x^i$ is the minimal polynomial expression of the zero function, therefore it is the zero polynomial by Proposition \ref{prop:unique_polynomial_expression}.
Hence, we have $\alpha_{i,0} = 0$, therefore $\alpha_{0,i} = 0$, for any index $i$.

Secondly, for any $x,y,z \in \mathbb{F}_p$, we have
\begin{equation}
\label{eq:associativity_right}
\begin{split}
(x_{\mathbb{Z}} \Ztimes y_{\mathbb{Z}}) \Ztimes z_{\mathbb{Z}}
&= \bigl( \psi_1(x,y)_{\mathbb{Z}} \Ztimes p \Zplus (xy)_{\mathbb{Z}} \bigr) \Ztimes z_{\mathbb{Z}} \\
&= \bigl( \psi_1(x,y)_{\mathbb{Z}} \Ztimes z_{\mathbb{Z}} \bigr) \Ztimes p \Zplus (xy)_{\mathbb{Z}} \Ztimes z_{\mathbb{Z}} \\
&\equiv \bigl( \psi_1(x,y) \cdot z \bigr)_{\mathbb{Z}} \Ztimes p \Zplus \psi_1(xy,z)_{\mathbb{Z}} \Ztimes p \Zplus ((xy)z)_{\mathbb{Z}} \pmod{p^2} \\
&\equiv \bigl( \psi_1(x,y) \cdot z + \psi_1(xy,z) \bigr)_{\mathbb{Z}} \Ztimes p \Zplus (xyz)_{\mathbb{Z}} \pmod{p^2} \enspace,
\end{split} 
\end{equation}
and similarly
\begin{equation}
\label{eq:associativity_left}
x_{\mathbb{Z}} \Ztimes (y_{\mathbb{Z}} \Ztimes z_{\mathbb{Z}})
\equiv \bigl( x \cdot \psi_1(y,z) + \psi_1(x,yz) \bigr)_{\mathbb{Z}} \Ztimes p \Zplus (xyz)_{\mathbb{Z}} \pmod{p^2} \enspace.
\end{equation}
By the associativity of multiplication, \eqref{eq:associativity_right} and \eqref{eq:associativity_left} are equal to each other.
Hence, by comparing the digits at the $p^1$'s places of \eqref{eq:associativity_right} and \eqref{eq:associativity_left}, we have
\begin{equation}
\label{eq:psi_almost_2-cocycle}
\psi_1(x,y) \cdot z + \psi_1(xy,z) = x \cdot \psi_1(y,z) + \psi_1(x,yz) \mbox{ for any } x,y,z \in \mathbb{F}_p \enspace,
\end{equation}
therefore, for any $x,y,z \in \mathbb{F}_p$, we have
\begin{equation}
\label{eq:relation_from_associativity}
\sum_{i,j=1}^{p-1} \alpha_{i,j} x^i y^j z + \sum_{i,j=1}^{p-1} \alpha_{i,j} x^i y^i z^j
= \sum_{i,j=1}^{p-1} \alpha_{i,j} x y^i z^j + \sum_{i,j=1}^{p-1} \alpha_{i,j} x^i y^j z^j \enspace.
\end{equation}
Since the degrees of the both sides with respect to each variable are at most $p-1$, Proposition \ref{prop:unique_polynomial_expression} implies that these are equivalent as polynomials.
Then, for $i,j \geq 2$ with $i \neq j$, by comparing the coefficients of $x^i y^j z$ in both sides of \eqref{eq:relation_from_associativity}, we have $\alpha_{i,j} = 0$.
On the other hand, for $i \geq 2$, by comparing the coefficients of $x^i y^i z$ in both sides of \eqref{eq:relation_from_associativity}, we have $\alpha_{i,i} + \alpha_{i,1} = 0$, therefore $\alpha_{i,1} = - \alpha_{i,i}$.
We also have $\alpha_{1,i} = - \alpha_{i,i}$ by the symmetry.
Summarizing the argument above, we have
\begin{equation}
\label{eq:expression_by_auxiliary_polynomial}
\begin{split}
\psi_1(x,y)
&= \alpha_{1,1} xy + \sum_{i=2}^{p-1} \alpha_{i,i} \bigl( x^iy^i - x^iy - xy^i \bigr) \\
&= xy \bigl( \Psi(xy) - \Psi(x) - \Psi(y) + \alpha_{1,1} \bigr) \enspace,
\end{split}
\end{equation}
where we define $\Psi(t) := \sum_{i=1}^{p-2} \alpha_{i+1,i+1} t^i$, which is a polynomial of degree at most $p-2$ with no constant term.
Now we have
\begin{displaymath}
0 = \psi_1(1,1) = \Psi(1) - \Psi(1) - \Psi(1) + \alpha_{1,1} = \alpha_{1,1} - \Psi(1) \enspace,
\end{displaymath}
therefore $\alpha_{1,1} = \Psi(1)$.
Hence Lemma \ref{lem:carry_for_p-ary_multiplication__step1} holds.
\end{proof}

\begin{lemma}
\label{lem:carry_for_p-ary_multiplication__step2}
In the situation of Theorem \ref{thm:carry_for_p-ary_multiplication}, for any $n \geq 1$, the function $\psi_1(x_1,\dots,x_n)$ can be written as \eqref{eq:carry_for_p-ary_multiplication__expression_by_Psi} for some polynomial $\Psi$ of degree at most $p-2$ with no constant term which is independent of $n$.
\end{lemma}
\begin{proof}
For the case $n = 1$, we have $\psi_1(x_1) = 0$ by the definition, while the right-hand side of \eqref{eq:carry_for_p-ary_multiplication__expression_by_Psi} becomes zero for an arbitrary choice of $\Psi$.
Therefore, the claim is trivial when $n = 1$.
The case $n = 2$ has been shown in Lemma \ref{lem:carry_for_p-ary_multiplication__step1}.
We prove the claim for the case $n \geq 3$ by induction.
We have
\begin{displaymath}
\begin{split}
(x_1)_{\mathbb{Z}} \Ztimes \cdots \Ztimes (x_{n-1})_{\mathbb{Z}} \Ztimes (x_n)_{\mathbb{Z}}
&= \bigl( (x_1)_{\mathbb{Z}} \Ztimes \cdots \Ztimes (x_{n-1})_{\mathbb{Z}} \bigr) \Ztimes (x_n)_{\mathbb{Z}} \\
&\equiv \bigl( \psi_1(x_1,\dots,x_{n-1})_{\mathbb{Z}} \Ztimes p \Zplus (x_1 \cdots x_{n-1})_{\mathbb{Z}} \bigr) \Ztimes (x_n)_{\mathbb{Z}} \pmod{p^2} \\
&= \psi_1(x_1,\dots,x_{n-1})_{\mathbb{Z}} \Ztimes (x_n)_{\mathbb{Z}} \Ztimes p \Zplus (x_1 \cdots x_{n-1})_{\mathbb{Z}} \Ztimes (x_n)_{\mathbb{Z}} \enspace.
\end{split}
\end{displaymath}
Now we have
\begin{displaymath}
\psi_1(x_1,\dots,x_{n-1})_{\mathbb{Z}} \Ztimes (x_n)_{\mathbb{Z}} \equiv \bigl( \psi_1(x_1,\dots,x_{n-1}) \cdot x_n \bigr)_{\mathbb{Z}} \pmod{p}
\end{displaymath}
and
\begin{displaymath}
(x_1 \cdots x_{n-1})_{\mathbb{Z}} \Ztimes (x_n)_{\mathbb{Z}} \equiv \psi_1(x_1 \cdots x_{n-1},x_n)_{\mathbb{Z}} \Ztimes p \Zplus (x_1 \cdots x_n)_{\mathbb{Z}} \pmod{p^2} \enspace.
\end{displaymath}
Since $a_{\mathbb{Z}} \Zplus b_{\mathbb{Z}} \equiv (a+b)_{\mathbb{Z}} \pmod{p}$ for any $a,b \in \mathbb{F}_p$, the combination of the equalities above implies that
\begin{displaymath}
\begin{split}
&(x_1)_{\mathbb{Z}} \Ztimes \cdots \Ztimes (x_{n-1})_{\mathbb{Z}} \Ztimes (x_n)_{\mathbb{Z}} \\
&\equiv \bigl( \psi_1(x_1,\dots,x_{n-1}) \cdot x_n + \psi_1( x_1 \cdots x_{n-1},x_n) \bigr)_{\mathbb{Z}} \Ztimes p \Zplus (x_1 \cdots x_n)_{\mathbb{Z}} \pmod{p^2} \enspace,
\end{split}
\end{displaymath}
therefore we have
\begin{displaymath}
\psi_1(x_1,\dots,x_n) = \psi_1(x_1,\dots,x_{n-1}) \cdot x_n + \psi_1( x_1 \cdots x_{n-1},x_n) \enspace.
\end{displaymath}
Now the induction hypothesis implies that the right-hand side is equal to
\begin{displaymath}
\begin{split}
&x_1 \cdots x_{n-1} \left( \Psi(x_1 \cdots x_{n-1}) - \sum_{j=1}^{n-1} \Psi(x_j) + (n-2) \Psi(1) \right) \cdot x_n \\
&\quad + (x_1 \cdots x_{n-1})x_n \bigl( \Psi((x_1 \cdots x_{n-1})x_n) - \Psi(x_1 \cdots x_{n-1}) - \Psi(x_n) + \Psi(1) \bigr) \\
&= x_1 \cdots x_n \left( \Psi(x_1 \cdots x_n) - \sum_{j=1}^{n} \Psi(x_j) + (n-1) \Psi(1) \right) \enspace,
\end{split}
\end{displaymath}
as desired.
Hence Lemma \ref{lem:carry_for_p-ary_multiplication__step2} holds.
\end{proof}

Before moving to the final step of the proof of Theorem \ref{thm:carry_for_p-ary_multiplication}, we note some properties of the Bernoulli polynomials $B_m(x)$, which is defined in terms of the Bernoulli numbers $B_{\ell} = B_{\ell}(0) \in \mathbb{Q}$ by
\begin{equation}
\label{eq:definition_of_Bernoulli_polynomial}
B_m(x) = \sum_{s=0}^{m} \binom{m}{s} B_{m-s} x^s \enspace.
\end{equation}
First, we note the following consequence of the von Staudt--Clausen Theorem (see e.g., Chapter 15 of \cite{IR}):

\begin{proposition}
\label{prop:Staudt-Clausen}
For any even integer $\ell > 0$, the denominator of $B_{\ell}$ is the product of all primes $q$ for which $q-1$ divides $\ell$.
\end{proposition}

By Proposition \ref{prop:Staudt-Clausen} and the fact that $B_0 = 1$, $B_1 = -1/2$ and $B_{\ell} = 0$ for every odd index $\ell > 1$, it follows that, for any odd prime $p$, the denominators of $B_0,B_1,\dots,B_{p-3}$ are all coprime to $p$.
Hence, Lemma \ref{lem:rational_number_as_element_of_Fp} can be applied to the Bernoulli polynomials $B_m(x)$ with $0 \leq m \leq p-3$.
In particular, for $a,b \in \mathbb{Q}$ with denominators being coprime to $p$, if $0 \leq m \leq p-3$ and $a^{\langle p \rangle} = b^{\langle p \rangle}$, then we have
\begin{displaymath}
B_m(a)^{\langle p \rangle}
= B_m{}^{\langle p \rangle}(a^{\langle p \rangle})
= B_m{}^{\langle p \rangle}(b^{\langle p \rangle})
= B_m(b)^{\langle p \rangle} \enspace.
\end{displaymath}
Secondly, it is known (see e.g., Chapter 15 of \cite{IR}) that, for any positive integers $m,N$, we have
\begin{equation}
\label{eq:power_sum_by_Bernoulli_polynomial}
\sum_{k=1}^{N} k^m = \frac{ 1 }{ m+1 } \bigl( B_{m+1}(N+1) - B_{m+1} \bigr) \enspace.
\end{equation}
Finally, we use the following property in the argument below (see e.g., Chapter 15 of \cite{IR}):

\begin{proposition}
\label{prop:multiplication_theorem}
For integers $m \geq 1$ and $n \geq 0$, we have
\begin{displaymath}
B_n(mx) = m^{n-1} \sum_{k=0}^{m-1} B_n\left( x + \frac{ k }{ m } \right) \enspace.
\end{displaymath}
\end{proposition}

\begin{proof}
[Proof of Theorem \ref{thm:carry_for_p-ary_multiplication}]
By Lemmas \ref{lem:carry_for_p-ary_multiplication__step1} and \ref{lem:carry_for_p-ary_multiplication__step2}, the remaining task is to show that the polynomial $\Psi(t) = \sum_{i=1}^{p-2} \beta_i t^i$ specified in Lemma \ref{lem:carry_for_p-ary_multiplication__step1} satisfies that $\beta_i = \left( B_{p-1-i} / (p-1-i) \right)^{\langle p \rangle}$ for every index $i$, and to show the relation $\Psi(1) = (w_p)^{\langle p \rangle}$ at the last of the statement.
We use the expression of $\psi_1(x,y)$ as in \eqref{eq:carry_for_p-ary_multiplication__expression_by_Psi} which has been proven in Lemma \ref{lem:carry_for_p-ary_multiplication__step1}.

Let $\xi$ be a primitive root modulo $p$.
Then for each index $1 \leq i \leq p-2$, the coefficient of $x^{i+1}$ in $\psi_1(x,\xi) = \xi x( \Psi(\xi x) - \Psi(x) - \Psi(\xi) + \Psi(1) )$ is $\beta_i \xi(\xi^i - 1)$.
On the other hand, for each integer $0 \leq k \leq \xi_{\mathbb{Z}}-1$, we have $\psi_1(x,\xi) = k$ if $\lceil kp / \xi_{\mathbb{Z}} \rceil \leq x_{\mathbb{Z}} \leq \lceil (k+1)p / \xi_{\mathbb{Z}} \rceil - 1$.
Therefore, we have
\begin{displaymath}
\psi_1(x,\xi) = \sum_{k=1}^{\xi_{\mathbb{Z}}-1} \sum_{z = \lceil kp / \xi_{\mathbb{Z}} \rceil}^{ \lceil (k+1)p / \xi_{\mathbb{Z}} \rceil - 1} k \cdot \chi[x = z]
= \sum_{k=1}^{\xi_{\mathbb{Z}}-1} \sum_{z = \lceil kp / \xi_{\mathbb{Z}} \rceil}^{ \lceil (k+1)p / \xi_{\mathbb{Z}} \rceil - 1} k \cdot (1 - (x-z)^{p-1}) \enspace.
\end{displaymath}
The coefficient (in $\mathbb{F}_p$) of $x^{i+1}$ in the right-hand side is
\begin{displaymath}
- \sum_{k=1}^{\xi_{\mathbb{Z}}-1} \sum_{z = \lceil kp / \xi_{\mathbb{Z}} \rceil}^{ \lceil (k+1)p / \xi_{\mathbb{Z}} \rceil - 1} k \binom{p-1}{i+1} (-z)^{p-i-2}
= - \sum_{k=1}^{\xi_{\mathbb{Z}}-1} \sum_{z = \lceil kp / \xi_{\mathbb{Z}} \rceil}^{ \lceil (k+1)p / \xi_{\mathbb{Z}} \rceil - 1} k z^{p-i-2}
\end{displaymath}
where we used the fact $\binom{p-1}{i+1} \equiv (-1)^{i+1} \pmod{p}$ (note that now $(-1)^{p-1} = 1$).
By the argument above, we have
\begin{displaymath}
\beta_i \xi (\xi^i - 1)
= - \sum_{k=1}^{\xi_{\mathbb{Z}}-1} \sum_{z = \lceil kp / \xi_{\mathbb{Z}} \rceil}^{ \lceil (k+1)p / \xi_{\mathbb{Z}} \rceil - 1} k z^{p-i-2} \enspace.
\end{displaymath}
For the right-hand side, we have
\begin{displaymath}
\begin{split}
\sum_{k=1}^{\xi_{\mathbb{Z}}-1} \sum_{z = \lceil kp / \xi_{\mathbb{Z}} \rceil}^{ \lceil (k+1)p / \xi_{\mathbb{Z}} \rceil - 1} k z^{p-i-2}
&= \sum_{k=1}^{\xi_{\mathbb{Z}}-1} k \left( \sum_{z = 1}^{ \lceil (k+1)p / \xi_{\mathbb{Z}} \rceil - 1} z^{p-i-2} - \sum_{z = 1}^{ \lceil kp / \xi_{\mathbb{Z}} \rceil - 1} z^{p-i-2} \right) \\
&= (\xi - 1) \sum_{z = 1}^{p-1} z^{p-i-2} - \sum_{k=1}^{\xi_{\mathbb{Z}} - 1} \sum_{z = 1}^{ \lceil kp / \xi_{\mathbb{Z}} \rceil - 1} z^{p-i-2} \enspace.
\end{split}
\end{displaymath}
To compute the first term of the right-hand side, we have the following equality in $\mathbb{F}_p$:
\begin{equation}
\label{eq:power_sum_in_prime_field}
\sum_{z \in \mathbb{F}_p \setminus \{0\}} z^j
= \sum_{\ell=0}^{p-2} (\xi^{\ell})^j
= \sum_{\ell=0}^{p-2} (\xi^j)^{\ell}
=
\begin{cases}
\displaystyle \frac{ (\xi^j)^{p-1} - 1 }{ \xi^j - 1 } = 0 & \mbox{(for $1 \leq j \leq p-2$)} \\
p-1 = -1 & \mbox{(for $j = 0$ and $j = p-1$)}
\end{cases}
\end{equation}
where we used the fact that $\xi^j \neq 1$ for $1 \leq j \leq p-2$ and $(\xi^j)^{p-1} = 1$ (by Fermat's Little Theorem).
Therefore, we have
\begin{equation}
\label{eq:proof_thm:carry_for_p-ary_multiplication__intermediate_1}
\beta_i \xi (\xi^i - 1) = \chi[i = p-2] \cdot (\xi-1) + \sum_{k=1}^{\xi_{\mathbb{Z}} - 1} \sum_{z = 1}^{ \lceil kp / \xi_{\mathbb{Z}} \rceil - 1} z^{p-i-2} \enspace.
\end{equation}

For the case $1 \leq i \leq p-3$, by applying the fact \eqref{eq:power_sum_by_Bernoulli_polynomial}, we have
\begin{displaymath}
\begin{split}
\left( \sum_{k=1}^{\xi_{\mathbb{Z}} - 1} \sum_{z = 1}^{ \lceil kp / \xi_{\mathbb{Z}} \rceil - 1} z^{p-i-2} \right)^{\langle p \rangle}
&= \left( \sum_{k=1}^{\xi_{\mathbb{Z}} - 1} \frac{ 1 }{ p-1-i } \left( B_{p-1-i}\left( \left\lceil \frac{ kp }{ \xi_{\mathbb{Z}} } \right\rceil \right) - B_{p-1-i} \right) \right)^{\langle p \rangle} \\
&= \left( \frac{ 1 }{ p-1-i } \right)^{\langle p \rangle} \left( \sum_{k=1}^{\xi_{\mathbb{Z}} - 1} B_{p-1-i}\left( \left\lceil \frac{ kp }{ \xi_{\mathbb{Z}} } \right\rceil \right)^{\langle p \rangle} - (\xi - 1) (B_{p-1-i})^{\langle p \rangle} \right) \enspace.
\end{split}
\end{displaymath}
For each index $1 \leq k \leq \xi_{\mathbb{Z}} - 1$, let $\delta_k$ denote the remainder (in the range $[\xi_{\mathbb{Z}} - 1]$) of $kp$ modulo $\xi_{\mathbb{Z}}$.
Then, since $\xi_{\mathbb{Z}}$ is coprime to $p$, $\delta_1$ is a generator of the additive cyclic group $\mathbb{Z}/\xi_{\mathbb{Z}} \mathbb{Z}$.
This implies that the $\delta_k$ are all distinct and $\{\delta_1,\delta_2,\dots,\delta_{\xi_{\mathbb{Z}} - 1}\} = \{1,2,\dots,\xi_{\mathbb{Z}} - 1\}$.
Moreover, for each index $1 \leq k \leq \xi_{\mathbb{Z}} - 1$, we have $\lceil kp / \xi_{\mathbb{Z}} \rceil = kp / \xi_{\mathbb{Z}} + ( \xi_{\mathbb{Z}} - \delta_k ) / \xi_{\mathbb{Z}}$ by the definition of $\delta_k$, therefore
\begin{equation}
\label{eq:proof_thm:carry_for_p-ary_multiplication__intermediate_2}
\left\lceil \frac{ kp }{ \xi_{\mathbb{Z}} } \right\rceil^{\langle p \rangle}
= \left( \frac{ kp }{ \xi_{\mathbb{Z}} } + \frac{ \xi_{\mathbb{Z}} - \delta_k }{ \xi_{\mathbb{Z}} } \right)^{\langle p \rangle}
= \left( \frac{ \xi_{\mathbb{Z}} - \delta_k }{ \xi_{\mathbb{Z}} } \right)^{\langle p \rangle} \enspace.
\end{equation}
This implies that
\begin{displaymath}
\begin{split}
\sum_{k=1}^{\xi_{\mathbb{Z}} - 1} B_{p-1-i}\left( \left\lceil \frac{ kp }{ \xi_{\mathbb{Z}} } \right\rceil \right)^{\langle p \rangle}
&= \sum_{k=1}^{\xi_{\mathbb{Z}} - 1} B_{p-1-i}\left( \frac{ \xi_{\mathbb{Z}} - \delta_k }{ \xi_{\mathbb{Z}} } \right)^{\langle p \rangle} \\
&= \sum_{k=1}^{\xi_{\mathbb{Z}} - 1} B_{p-1-i}\left( \frac{ \xi_{\mathbb{Z}} - k }{ \xi_{\mathbb{Z}} } \right)^{\langle p \rangle}
= \sum_{k=1}^{\xi_{\mathbb{Z}} - 1} B_{p-1-i}\left( \frac{ k }{ \xi_{\mathbb{Z}} } \right)^{\langle p \rangle} \enspace,
\end{split}
\end{displaymath}
therefore
\begin{displaymath}
\sum_{k=1}^{\xi_{\mathbb{Z}} - 1} B_{p-1-i}\left( \left\lceil \frac{ kp }{ \xi_{\mathbb{Z}} } \right\rceil \right)^{\langle p \rangle} - (\xi - 1) (B_{p-1-i})^{\langle p \rangle}
= \sum_{k=0}^{\xi_{\mathbb{Z}} - 1} B_{p-1-i}\left( \frac{ k }{ \xi_{\mathbb{Z}} } \right)^{\langle p \rangle} - \xi (B_{p-1-i})^{\langle p \rangle} \enspace.
\end{displaymath}
Moreover, by setting $x = 0$, $m = \xi_{\mathbb{Z}}$ and $n = p-1-i$ in Proposition \ref{prop:multiplication_theorem}, it follows that
\begin{displaymath}
B_{p-1-i} = (\xi_{\mathbb{Z}})^{p-2-i} \sum_{k=0}^{\xi_{\mathbb{Z}}-1} B_{p-1-i}\left( \frac{ k }{ \xi_{\mathbb{Z}} } \right) \enspace,
\end{displaymath}
therefore (since $\xi^{p-1} = 1$ in $\mathbb{F}_p$)
\begin{displaymath}
\sum_{k=0}^{\xi_{\mathbb{Z}}-1} B_{p-1-i}\left( \frac{ k }{ \xi_{\mathbb{Z}} } \right)^{\langle p \rangle}
= \xi^{i+1} (B_{p-1-i})^{\langle p \rangle} \enspace.
\end{displaymath}
Summarizing, the right-hand side of \eqref{eq:proof_thm:carry_for_p-ary_multiplication__intermediate_1} is equal to
\begin{displaymath}
\left( \frac{ 1 }{ p-1-i } \right)^{\langle p \rangle} \left( \xi^{i+1} (B_{p-1-i})^{\langle p \rangle} - \xi (B_{p-1-i})^{\langle p \rangle} \right)
= \xi(\xi^i - 1) \left( \frac{ B_{p-1-i} }{ p-1-i } \right)^{\langle p \rangle} \enspace,
\end{displaymath}
therefore (since $\xi \neq 0$ and $\xi^i \neq 1$ by the choice of $\xi$) we have $\beta_i = \left( B_{p-1-i} / (p-1-i) \right)^{\langle p \rangle}$ as desired.

On the other hand, for the case $i = p-2$, we have
\begin{displaymath}
\left( \xi - 1 + \sum_{k=1}^{\xi_{\mathbb{Z}} - 1} \sum_{z = 1}^{ \lceil kp / \xi_{\mathbb{Z}} \rceil - 1} z^{p-i-2} \right)^{\langle p \rangle}
= \xi - 1 + \sum_{k=1}^{\xi_{\mathbb{Z}} - 1} \left( \left\lceil \frac{ kp }{ \xi_{\mathbb{Z}} } \right\rceil - 1 \right)^{\langle p \rangle}
= \sum_{k=1}^{\xi_{\mathbb{Z}} - 1} \left( \frac{ \xi_{\mathbb{Z}} - \delta_k }{ \xi_{\mathbb{Z}} } \right)^{\langle p \rangle}
\end{displaymath}
where we used the property \eqref{eq:proof_thm:carry_for_p-ary_multiplication__intermediate_2}.
Since $\{\delta_1,\delta_2,\dots,\delta_{\xi_{\mathbb{Z}}-1}\} = \{1,2,\dots,\xi_{\mathbb{Z}}-1\}$ as shown above, we have
\begin{displaymath}
\sum_{k=1}^{\xi_{\mathbb{Z}} - 1} \left( \frac{ \xi_{\mathbb{Z}} - \delta_k }{ \xi_{\mathbb{Z}} } \right)^{\langle p \rangle}
= \left( \sum_{k=1}^{\xi_{\mathbb{Z}} - 1} \frac{ k }{ \xi_{\mathbb{Z}} } \right)^{\langle p \rangle}
= \left( \frac{ \xi_{\mathbb{Z}} - 1 }{ 2 } \right)^{\langle p \rangle} \enspace.
\end{displaymath}
Hence, by \eqref{eq:proof_thm:carry_for_p-ary_multiplication__intermediate_1} and the fact $\xi^{p-1} = 1$, we have
\begin{displaymath}
\beta_{p-2} (1 - \xi) = \left( \frac{ \xi_{\mathbb{Z}} - 1 }{ 2 } \right)^{\langle p \rangle} \enspace,
\end{displaymath}
therefore, since $\xi \neq 1$ and $B_1 = -1/2$, we have $\beta_{p-2} = (-1/2)^{\langle p \rangle} = B_1{}^{\langle p \rangle}$, as desired.
Summarizing, the equality \eqref{eq:carry_for_p-ary_multiplication__Psi} is now proven.

Finally, we show that $\Psi(1) = (w_p)^{\langle p \rangle}$.
By using the relation \eqref{eq:carry_for_p-ary_multiplication__expression_by_Psi} with $n = p$, for any $x \in \mathbb{F}_p \setminus \{0\}$, we have (in $\mathbb{F}_p$)
\begin{displaymath}
\psi_1(\underbrace{x,\dots,x}_{p})
= x^p ( \Psi(x^p) - p \cdot \Psi(x) + (p-1) \Psi(1) )
= x ( \Psi(x) - \Psi(1) ) \enspace.
\end{displaymath}
This implies that
\begin{displaymath}
(x_{\mathbb{Z}})^p
\equiv (x^p)_{\mathbb{Z}} \Zplus \psi_1(\underbrace{x,\dots,x}_{p})_{\mathbb{Z}} \Ztimes p
\equiv x_{\mathbb{Z}} \Zplus x_{\mathbb{Z}} \Ztimes (\Psi(x) - \Psi(1))_{\mathbb{Z}} \Ztimes p \pmod{p^2} \enspace,
\end{displaymath}
therefore
\begin{equation}
\label{eq:rem:relation_to_group_cohomology__Fermat_quotient}
\Psi(x) - \Psi(1) = \left( \frac{ (x_{\mathbb{Z}})^{p-1} - 1 }{ p } \right)^{\langle p \rangle} = q_p(x_{\mathbb{Z}})^{\langle p \rangle}
\end{equation}
where $q_p(x) = (x^{p-1} - 1)/p$ denotes the Fermat quotient.
We use the following relation between the Fermat quotient and Wilson's quotient \cite{Ler}:
\begin{displaymath}
\sum_{a=1}^{p-1} q_p(a) \equiv w_p \pmod{p} \enspace.
\end{displaymath}
By this relation, we have
\begin{displaymath}
\begin{split}
w_p \equiv \sum_{x=1}^{p-1} q_p(x)
\equiv \sum_{x=1}^{p-1} (\Psi(x) - \Psi(1))
= \sum_{i=1}^{p-2} \beta_i \sum_{x=1}^{p-1} x^i - (p-1) \Psi(1)
\equiv 0 + \Psi(1) = \Psi(1) \pmod{p}
\end{split}
\end{displaymath}
as desired, where we used the equality \eqref{eq:power_sum_in_prime_field}.
This completes the proof of Theorem \ref{thm:carry_for_p-ary_multiplication}.
\end{proof}

We note that the minimal polynomial expression of a general function $(\mathbb{F}_p)^n \to \mathbb{F}_p$ consists of $p^n$ monomials in the worst case.
In contrast, the polynomial expression of $\psi_1$ given above consists of only 
$(n+1)(p-1)/2 + 1$ monomials, which is significantly fewer than the worst-case number $p^n$ of monomials.

\begin{remark}
\label{rem:relation_to_group_cohomology}
The expression \eqref{eq:carry_for_p-ary_multiplication__expression_by_Psi} of $\psi_1$ in terms of the auxiliary function $\Psi$ and a \lq\lq meaning'' of $\Psi$ can be interpreted from a more algebraic viewpoint.
See the Appendix below for the detailed observation.
\end{remark}

\begin{example}
\label{ex:carry_for_p-ary_multiplication__2_is_primitive}
We compute the polynomials $\Psi(t)$ and $\psi_1(x,y)$ for some small odd primes $p$.
For the case $p = 3$, $\Psi(t)$ has only the highest term $\Psi(t) = (p-1)/2 \cdot t^{p-2} = t$, therefore
\begin{displaymath}
\psi_1(x,y) = xy(xy - x - y + 1) = x(x-1)y(y-1) \mbox{ for } p = 3 \enspace.
\end{displaymath}
For the other $p$, we quote from A000367 and A002445 of \cite{OEIS} some values of Bernoulli numbers (Table \ref{tab:Bernoulli-number}), and from A002068 of \cite{OEIS} some values of Wilson's quotients; the polynomials $\Psi(t)$ are then calculated by using Theorem \ref{thm:carry_for_p-ary_multiplication} and Tables \ref{tab:Bernoulli-number} and \ref{tab:Wilson's_quotient}.

\begin{table}
\centering
\caption{Some Bernoulli numbers $B_{\ell}$; note that $B_{\ell} = 0$ for odd indices $\ell > 1$}
\label{tab:Bernoulli-number}
\begin{tabular}{|c|c|c|c|c|c|c|c|c|c|c|}
$\ell$ & $0$ & $1$ & $2$ & $4$ & $6$ & $8$ & $10$ & $12$ & $14$ & $16$ \\ \hline
$B_{\ell}$ & $1$ & $-1/2$ & $1/6$ & $-1/30$ & $1/42$ & $-1/30$ & $5/66$ & $-691/2730$ & $7/6$ & $-3617/510$ \\
$B_{\ell} / \ell$ & & $-1/2$ & $1/12$ & $-1/120$ & $1/252$ & $-1/240$ & $1/132$ & $-691/32760$ & $1/12$ & $-3617/8160$
\end{tabular}
\end{table}
\begin{table}
\centering
\caption{Some Wilson's quotients $w_p$ modulo primes $p$; recall that $\Psi(1) \equiv w_p \pmod{p}$}
\label{tab:Wilson's_quotient}
\begin{tabular}{|c|c|c|c|c|c|c|c|c|c|c|}
$p$ & $3$ & $5$ & $7$ & $11$ & $13$ & $17$ & $19$ & $23$ & $29$ & $31$ \\ \hline
$w_p \bmod p$ & $1$ & $0$ & $5$ & $1$ & $0$ & $5$ & $2$ & $8$ & $18$ & $19$
\end{tabular}
\end{table}

For $p = 5$, we have
\begin{displaymath}
\Psi(t)
= \left( 2t^3 + \frac{ 1 }{ 12 } t^2 \right)^{\langle 5 \rangle}
= 2t^3 + 3t^2 \,,\, \Psi(1) = 0 \,,\,
\psi_1(x,y) = xy( \Psi(xy) - \Psi(x) - \Psi(y) ) \enspace.
\end{displaymath}

For $p = 7$, we have
\begin{displaymath}
\begin{split}
\Psi(t)
&= \left( 3t^5 + \frac{ 1 }{ 12 } t^4 - \frac{ 1 }{ 120 } t^2 \right)^{\langle 7 \rangle}
= 3t^5 + 3t^4 - t^2 \enspace,\\
\Psi(1) &= 5 \,,\,
\psi_1(x,y) = xy( \Psi(xy) - \Psi(x) - \Psi(y) + 5 ) \enspace.
\end{split}
\end{displaymath}

For $p = 11$, we have
\begin{displaymath}
\begin{split}
\Psi(t)
&= \left( 5t^9 + \frac{ 1 }{ 12 } t^8 - \frac{ 1 }{ 120 } t^6 + \frac{ 1 }{ 252 } t^4 - \frac{ 1 }{ 240 } t^2 \right)^{\langle 11 \rangle}
= 5t^9 + t^8 + t^6 - t^4 - 5t^2 \enspace,\\
\Psi(1) &= 1 \,,\,
\psi_1(x,y) = xy( \Psi(xy) - \Psi(x) - \Psi(y) + 1 ) \enspace.
\end{split}
\end{displaymath}

For $p = 13$, we have
\begin{displaymath}
\begin{split}
\Psi(t)
&= \left( 6t^{11} + \frac{ 1 }{ 12 } t^{10} - \frac{ 1 }{ 120 } t^8 + \frac{ 1 }{ 252 } t^6 - \frac{ 1 }{ 240 } t^4 + \frac{ 1 }{ 132 } t^2 \right)^{\langle 13 \rangle}
 = 6t^{11} - t^{10} + 4t^8 - 5t^6 + 2t^4 - 6t^2 \enspace,\\
\Psi(1) &= 0 \,,\,
\psi_1(x,y) = xy( \Psi(xy) - \Psi(x) - \Psi(y) ) \enspace.
\end{split}
\end{displaymath}

For $p = 17$, we have
\begin{displaymath}
\begin{split}
\Psi(t)
&= \left( 8t^{15} + \frac{ 1 }{ 12 } t^{14} - \frac{ 1 }{ 120 } t^{12} + \frac{ 1 }{ 252 } t^{10} - \frac{ 1 }{ 240 } t^8 + \frac{ 1 }{ 132 } t^6 - \frac{ 691 }{ 32760 } t^4 + \frac{ 1 }{ 12 } t^2 \right)^{\langle 17 \rangle} \\
&= 8t^{15} - 7t^{14} - t^{12} - 6t^{10} + 8t^8 + 4t^6 + 6t^4 - 7t^2 \enspace,\\
\Psi(1) &= 5 \,,\,
\psi_1(x,y) = xy( \Psi(xy) - \Psi(x) - \Psi(y) + 5 ) \enspace.
\end{split}
\end{displaymath}

For $p = 19$, we have
\begin{displaymath}
\begin{split}
\Psi(t)
&= \left( 9t^{17} + \frac{ 1 }{ 12 } t^{16} - \frac{ 1 }{ 120 } t^{14} + \frac{ 1 }{ 252 } t^{12} - \frac{ 1 }{ 240 } t^{10} + \frac{ 1 }{ 132 } t^8 - \frac{ 691 }{ 32760 } t^6 + \frac{ 1 }{ 12 } t^4 - \frac{ 3617 }{ 8160 } t^2 \right)^{\langle 19 \rangle} \\
&= 9t^{17} + 8t^{16} + 3t^{14} + 4t^{12} - 8t^{10} - t^8 + 3t^6 + 8t^4 - 5t^2 \enspace,\\
\Psi(1) &= 2 \,,\,
\psi_1(x,y) = xy( \Psi(xy) - \Psi(x) - \Psi(y) + 2 ) \enspace.
\end{split}
\end{displaymath}
\end{example}

\subsection{Multiplication of $p$-ary Integers Based on Polynomials}
\label{subsec:polynomial_for_multiplication__algorithm}

Here we show two algorithms for multiplication of two $p$-ary integers $a_h = (a_{h,m_h} \cdots a_{h,1}a_{h,0})_p$, $h = 1,2$, based on the result of Section \ref{subsec:polynomial_for_multiplication__carry}, where, as above, each digit $a_{h,i}$ of $a_h$ is represented by an element of $\mathbb{F}_p$.
The advantage of the first algorithm is that we need the carry function $\varphi_1$ to the next digit for addition but do not need the carry functions $\varphi_k$ to higher digits $k \geq 2$ which are more complicated.
On the other hand, the advantage of the second algorithm is that it seems more appropriate for parallel computation.
As in Section \ref{subsec:polynomial_for_multiplication__carry}, we assume $p > 2$.

For our first algorithm, note that the product $c = a_1a_2$ can be expressed by $m_1 + m_2 + 2$ digits; $c = (c_{m_1 + m_2 + 1} \cdots c_1c_0)_p$, $c_i \in \mathbb{F}_p$.
Then the digits of $c$ are calculated by the algorithm shown in Figure \ref{fig:algorithm_multiplication_1}, where $\gamma$ means an auxiliary variable for the carry at each digit to the next digit.
We note that, for each indices $i,j$, we have
\begin{displaymath}
(a_{1,i})_{\mathbb{Z}} \Ztimes (a_{2,j})_{\mathbb{Z}} \Zplus (c_{i+j})_{\mathbb{Z}} \Zplus \gamma_{\mathbb{Z}} \leq (p-1)^2 + 2(p-1) = p^2 - 1 \enspace,
\end{displaymath}
therefore the value appearing in updating the $(i+j)$-th digit can be expressed by two digits and the polynomials $\varphi_k$ for $k \geq 2$ are not needed.
Now it follows that the algorithm calculates $c = a_1a_2$ correctly.

\begin{figure}
\centering
\caption{First algorithm for multiplication of two $p$-ary integers based on polynomials}\medskip
\label{fig:algorithm_multiplication_1}
\fbox{%
\begin{minipage}{400pt}
\noindent
\texttt{Input:} $a_h = (a_{h,m_h} \cdots a_{h,1}a_{h,0})_p$ ($h \in \{1,2\}$, $a_{h,i} \in \mathbb{F}_p$)\medskip

\noindent
\texttt{Set} $c_0 \leftarrow a_{1,0}a_{2,0}$, $\gamma \leftarrow \psi_1(a_{1,0},a_{2,0})$ \\
\texttt{For} $i = 1,\dots,m_1$ \texttt{Do:} \\
\hspace*{1em} \texttt{Set} $c_i \leftarrow a_{1,i}a_{2,0} + \gamma$ \\
\hspace*{1em} \texttt{Update} $\gamma$ \texttt{by} $\gamma \leftarrow \psi_1(a_{1,i},a_{2,0}) + \varphi_1(a_{1,i}a_{2,0},\gamma)$ \\
\texttt{End Do} \\
\texttt{Set} $c_{1,m_1+1} \leftarrow \gamma$ \\
\texttt{For} $j = 1,\dots,m_2$ \texttt{Do:} \\
\hspace*{1em} \texttt{Update} $c_j$ \texttt{and} $\gamma$ \texttt{by} $(c_j,\gamma) \leftarrow \bigl( a_{1,0}a_{2,j} + c_j, \psi_1(a_{1,0},a_{2,j}) + \varphi_1(a_{1,0}a_{2,j},c_j) \bigr)$ \\
\hspace*{1em} \texttt{For} $i = 1,\dots,m_1-1$ \texttt{Do:} \\
\hspace*{2em} \texttt{Update} $c_{i+j}$ \texttt{and} $\gamma$ \texttt{by} $(c_{i+j},\gamma) \leftarrow \bigl( a_{1,i}a_{2,j} + c_{i+j} + \gamma, \psi_1(a_{1,i},a_{2,j}) + \varphi_1(a_{1,i}a_{2,j},c_{i+j},\gamma) \bigr)$ \\
\hspace*{1em} \texttt{End Do} \\
\hspace*{1em} \texttt{Update} $c_{m_1+j}$ \texttt{by} $c_{m_1+j} \leftarrow a_{1,m_1}a_{2,j} + c_{m_1+j} + \gamma$ \\
\hspace*{1em} \texttt{Set} $c_{m_1+j+1} \leftarrow \psi_1(a_{1,m_1},a_{2,j}) + \varphi_1(a_{1,m_1}a_{2,j},c_{m_1+j},\gamma)$ \\
\texttt{End Do} \\
\texttt{Output} $c = (c_{m_1+m_2+1} \cdots c_1c_0)_p$
\end{minipage}
}
\end{figure}
On the other hand, our second algorithm to calculate the digits of $c = a_1 a_2$ is shown in Figure \ref{fig:algorithm_multiplication_2}.
Here we note that, for the latter loop for $i = 0,1,\dots$, since we have $n(p-1) < p^n$ for any integer $n \geq 1$ and any prime $p$, the total number of elements in the lists $A_k$ with $k \geq i$ is strictly decreasing when $i$ is incremented during the loop.
This implies that the algorithm always stops within a finite number of steps, therefore the algorithm calculates $c = a_1 a_2$ correctly.

\begin{figure}
\centering
\caption{Second algorithm for multiplication of two $p$-ary integers based on polynomials}\medskip
\label{fig:algorithm_multiplication_2}
\fbox{%
\begin{minipage}{400pt}
\noindent
\texttt{Input:} $a_h = (a_{h,m_h} \cdots a_{h,1}a_{h,0})_p$ ($h \in \{1,2\}$, $a_{h,i} \in \mathbb{F}_p$)\medskip

\noindent
\texttt{Initialize the lists $A_0,A_1,A_2,\dots$ to be empty} \\
\texttt{For} $i = 0,\dots,m_1$ \texttt{Do:} \\
\hspace*{1em} \texttt{For} $j = 0,\dots,m_2$ \texttt{Do:} \\
\hspace*{2em} \texttt{Append} $a_{1,i}a_{2,j}$ \texttt{to the list} $A_{i+j}$ \\
\hspace*{2em} \texttt{Append} $\psi_1(a_{1,i},a_{2,j})$ \texttt{to the list} $A_{i+j+1}$ \\
\hspace*{1em} \texttt{End Do} \\
\texttt{End Do} \\
\texttt{For} $i = 0,1,\dots$ \texttt{Do} \\
\hspace*{1em} \texttt{If} $A_i$ \texttt{is empty, then output} $c = (c_{i-1} \dots c_1c_0)_p$ \texttt{and stop} \\
\hspace*{1em} \texttt{Enumerate the elements of} $A_i$ \texttt{as} $\alpha_1,\dots,\alpha_n$ \\
\hspace*{1em} \texttt{Set} $c_i \leftarrow \alpha_1 + \cdots + \alpha_n$ \\
\hspace*{1em} \texttt{For} $j = 1,\dots,\max\{k \in \mathbb{Z} \mid n(p-1) \geq p^k\}$ \texttt{Do:} \\
\hspace*{2em} \texttt{Append} $\varphi_j(\alpha_1,\dots,\alpha_n)$ \texttt{to the list} $A_{i+j}$ \\
\hspace*{1em} \texttt{End Do} \\
\texttt{End Do}
\end{minipage}
}
\end{figure}

\appendix

\section*{Appendix: Algebraic Observation for the Proof of Theorem \ref{thm:carry_for_p-ary_multiplication}}

In this appendix, we revisit our proof of Theorem \ref{thm:carry_for_p-ary_multiplication} from algebraic viewpoints, as mentioned in Remark \ref{rem:relation_to_group_cohomology}.

Let $p$ be an odd prime.
First, we consider the following exact sequence
\begin{displaymath}
1 \to 1 + p\mathbb{Z}/p^2\mathbb{Z} \hookrightarrow (\mathbb{Z}/p^2{\mathbb{Z}})^{\times} \overset{\bmod p}{\to} (\mathbb{F}_p)^{\times} \to 1
\end{displaymath}
and a section $\widetilde{\cdot} \colon (\mathbb{F}_p)^{\times} \ni x \mapsto \widetilde{x} \in (\mathbb{Z}/p^2{\mathbb{Z}})^{\times}$ which is a composition of the map $a \mapsto a_{\mathbb{Z}}$ followed by the natural projection $\mathbb{Z} \twoheadrightarrow \mathbb{Z}/p^2\mathbb{Z}$.
Note that the group action of $(\mathbb{F}_p)^{\times}$ on $1 + p\mathbb{Z}/p^2\mathbb{Z}$ associated to the group extension above is trivial, since $(\mathbb{Z}/p^2{\mathbb{Z}})^{\times}$ is Abelian.
Then, by the general theory of cohomology of groups, the map $(\mathbb{F}_p)^{\times} \times (\mathbb{F}_p)^{\times} \to (\mathbb{Z}/p^2\mathbb{Z})^{\times}$, $(x,y) \mapsto \widetilde{x}\,\widetilde{y} / \widetilde{xy}$, has values in the subgroup $1 + p\mathbb{Z}/p^2\mathbb{Z}$ and gives a $2$-cocycle, hence an element of $H^2((\mathbb{F}_p)^{\times},1 + p\mathbb{Z}/p^2\mathbb{Z})$.
Since $x_{\mathbb{Z}}y_{\mathbb{Z}} = (xy)_{\mathbb{Z}} + \psi_1(x,y)_{\mathbb{Z}} \cdot p$, we have $\widetilde{x}\,\widetilde{y} / \widetilde{xy} = 1 + \left( \psi_1(x,y)_{\mathbb{Z}} / (xy)_{\mathbb{Z}} \right)^{\langle p^2 \rangle} \cdot p$.
By mapping this via a group isomorphism $1 + p\mathbb{Z}/p^2\mathbb{Z} \overset{\sim}{\to} \mathbb{Z}/p\mathbb{Z}$, $a \mapsto (a-1)/p$, we obtain a $2$-cocycle $(\mathbb{F}_p)^{\times} \times (\mathbb{F}_p)^{\times} \to \mathbb{Z}/p\mathbb{Z}$ given by
\begin{equation}
\label{eq:rem:relation_to_group_cohomology__2-cocycle}
(\mathbb{F}_p)^{\times} \times (\mathbb{F}_p)^{\times} \ni (x,y) \mapsto \left( \frac{ \psi_1(x,y)_{\mathbb{Z}} }{ (xy)_{\mathbb{Z}} } \right)^{\langle p \rangle} = \frac{ \psi_1(x,y) }{ xy } \in \mathbb{Z}/p\mathbb{Z} \enspace.
\end{equation}
The property \eqref{eq:psi_almost_2-cocycle} for $x,y,z \in (\mathbb{F}_p)^{\times}$ is now derived by the definition of $2$-cocycles (for the trivial group action).
We note that the property \eqref{eq:psi_almost_2-cocycle} for the remaining case where some of $x,y,z$ is zero follows immediately from the meaning of $\psi_1$.
Moreover, since $(\mathbb{F}_p)^{\times}$ and $\mathbb{Z}/p\mathbb{Z}$ have coprime orders, we have $H^2((\mathbb{F}_p)^{\times},\mathbb{Z}/p\mathbb{Z}) = 0$ by Schur--Zassenhaus Theorem.
In particular, the $2$-cocycle \eqref{eq:rem:relation_to_group_cohomology__2-cocycle} gives a zero element of $H^2((\mathbb{F}_p)^{\times},\mathbb{Z}/p\mathbb{Z})$ and hence is a coboundary (for the trivial group action), namely,
\begin{equation}
\label{eq:rem:relation_to_group_cohomology__2-cocycle_is_coboundary}
\frac{ \psi_1(x,y) }{ xy } = \overline{\Psi}(x) + \overline{\Psi}(y) - \overline{\Psi}(xy)
\end{equation}
for a function $\overline{\Psi} \colon (\mathbb{F}_p)^{\times} \to \mathbb{Z}/p\mathbb{Z}$.
Now we have $\overline{\Psi}(1) = \psi_1(1,1) = 0$.
Then the expression \eqref{eq:carry_for_p-ary_multiplication__expression_by_Psi} of $\psi_1$ for $n = 2$ is deduced by extending the domain of the function $\overline{\Psi}$ from $(\mathbb{F}_p)^{\times}$ to $\mathbb{F}_p$ and normalizing it in such a way that $\Psi(t) = \overline{\Psi}(0) - \overline{\Psi}(t)$, i.e., $\Psi(1) = \overline{\Psi}(0)$ and $\overline{\Psi}(t) = \Psi(1) - \Psi(t)$.
We note that such a function $(\mathbb{F}_p)^{\times} \to \mathbb{Z}/p\mathbb{Z}$ satisfying \eqref{eq:rem:relation_to_group_cohomology__2-cocycle_is_coboundary} is uniquely determined.
Indeed, the $1$-cocycles $(\mathbb{F}_p)^{\times} \to \mathbb{Z}/p\mathbb{Z}$ are group homomorphisms since $(\mathbb{F}_p)^{\times}$ acts trivially on $\mathbb{Z}/p\mathbb{Z}$, while we have $\mathrm{Hom}((\mathbb{F}_p)^{\times},\mathbb{Z}/p\mathbb{Z}) = 0$ since $(\mathbb{F}_p)^{\times}$ and $\mathbb{Z}/p\mathbb{Z}$ have coprime orders.
Therefore, the difference of any two such functions, which is a $1$-cocycle, is the zero map as mentioned above.

To investigate the function $\overline{\Psi}$ further, we consider another section $[\cdot] \colon (\mathbb{F}_p)^{\times} \to (\mathbb{Z}/p^2\mathbb{Z})^{\times}$ defined by $[x] = (\widetilde{x})^p$ (note that $[x] \equiv x^p \equiv x \pmod{p}$ by Fermat's Little Theorem).
This is a group homomorphism (hence, it is the Teichm\"{u}ller lift of the projection $(\mathbb{Z}/p^2\mathbb{Z})^{\times} \to (\mathbb{F}_p)^{\times}$), since $\widetilde{x}\,\widetilde{y} \equiv \widetilde{xy} + \psi_1(x,y) \cdot p \pmod{p^2}$ and hence $(\widetilde{x}\,\widetilde{y})^p \equiv (\widetilde{xy})^p \pmod{p^2}$ by the binomial theorem.
We consider the difference $\widetilde{x} [x]^{-1} \in 1 + p\mathbb{Z}/p^2\mathbb{Z}$ of the two sections $\widetilde{\cdot},[\cdot]$.
By mapping this via the isomorphism $1 + p\mathbb{Z}/p^2\mathbb{Z} \overset{\sim}{\to} \mathbb{Z}/p\mathbb{Z}$ above, we obtain the map
\begin{displaymath}
\alpha \colon (\mathbb{F}_p)^{\times} \to \mathbb{Z}/p\mathbb{Z} \,,\, \alpha(x) = \left( \frac{ \widetilde{x} [x]^{-1} - 1 }{ p } \right)^{\langle p \rangle} \enspace.
\end{displaymath}
Now, by the homomorphic property of $[\cdot]$, for any $x,y \in (\mathbb{F}_p)^{\times}$, we have
\begin{displaymath}
\begin{split}
\psi_1(x,y)
&= \left( \frac{ \widetilde{x}\,\widetilde{y} - \widetilde{xy} }{ p } \right)^{\langle p \rangle} \\
&= \left( [xy] \cdot \frac{ (\widetilde{x}\,\widetilde{y} [x]^{-1}[y]^{-1} - 1) - (\widetilde{xy} [xy]^{-1} - 1) }{ p } \right)^{\langle p \rangle} \\
&= \left( [xy] \cdot \frac{ (\widetilde{x} [x]^{-1} - 1)(\widetilde{y} [y]^{-1} - 1) + (\widetilde{x} [x]^{-1} - 1) + (\widetilde{y} [y]^{-1} - 1 ) - (\widetilde{xy} [xy]^{-1} - 1) }{ p } \right)^{\langle p \rangle} \enspace.
\end{split}
\end{displaymath}
Since $\widetilde{x} [x]^{-1} - 1 \equiv \widetilde{y} [y]^{-1} - 1 \equiv 0 \pmod{p}$, the rightmost side is equal to
\begin{displaymath}
\left( [xy] \cdot \frac{ (\widetilde{x} [x]^{-1} - 1) + (\widetilde{y} [y]^{-1} - 1 ) - (\widetilde{xy} [xy]^{-1} - 1) }{ p } \right)^{\langle p \rangle}
= xy (\alpha(x) + \alpha(y) - \alpha(xy)) \enspace,
\end{displaymath}
therefore $\psi_1(x,y) / (xy) = \alpha(x) + \alpha(y) - \alpha(xy)$.
Hence we have
\begin{displaymath}
\overline{\Psi}(x) = \alpha(x) = \left( \frac{ \widetilde{x} [x]^{-1} - 1 }{ p } \right)^{\langle p \rangle} \mbox{\quad for $x \in (\mathbb{F}_p)^{\times}$}
\end{displaymath}
by the uniqueness of $\overline{\Psi}$ mentioned above.
This gives a \lq\lq meaning'' of the auxiliary function $\overline{\Psi}$ (and its normalized version $\Psi$) as the difference of the two sections $\widetilde{\cdot}$ and $[\cdot]$ in the group extension above.

For any $a \in \mathbb{F}_p \setminus \{0,-1\}$, we have $\widetilde{a+1} = \widetilde{a} + 1$ and
\begin{displaymath}
\begin{split}
(a+1) \overline{\Psi}(a+1)
&= \left( \frac{ \widetilde{a+1} - [a+1] }{ p } \right)^{\langle p \rangle} \\
&= \left( \frac{ (\widetilde{a} - [a]) + ([a] + 1 - [a+1]) }{ p } \right)^{\langle p \rangle}
= a \overline{\Psi}(a) + \left( \frac{ [a] + 1 - [a+1] }{ p } \right)^{\langle p \rangle} \enspace,
\end{split}
\end{displaymath}
therefore
\begin{equation}
\label{eq:rem:relation_to_group_cohomology__differential_equation}
(a+1) \overline{\Psi}(a+1) - a \overline{\Psi}(a)
= \left( \frac{ [a] + 1 - [a+1] }{ p } \right)^{\langle p \rangle} \enspace.
\end{equation}
Intuitively, the differential equation \eqref{eq:rem:relation_to_group_cohomology__differential_equation} involving the power function $[a] = (\widetilde{a})^p$ can be seen as the source of Bernoulli numbers appearing in the expression of $\psi_1$, since Bernoulli numbers have close connections to power sums (cf., \eqref{eq:power_sum_by_Bernoulli_polynomial}).
Now for $x \in (\mathbb{F}_p)^{\times}$, by summing up \eqref{eq:rem:relation_to_group_cohomology__differential_equation} for $a \in 1,2,\dots,x-1$ and by using the fact $\overline{\Psi}(1) = 0$, we have
\begin{displaymath}
x \overline{\Psi}(x) = \left( \frac{ x - [x] }{p} \right)^{\langle p \rangle} = -x \cdot q_p(x)^{\langle p \rangle}
\end{displaymath}
where $q_p(x) = (x^{p-1} - 1)/p$ denotes the Fermat quotient.
Hence, the relation \eqref{eq:rem:relation_to_group_cohomology__Fermat_quotient} of the auxiliary function $\Psi$ to the Fermat quotient can be derived from the \lq\lq meaning'' of $\Psi$ itself mentioned above, without using the original function $\psi_1$.

\end{document}